\documentclass{amsart}
\usepackage{mathrsfs}
\usepackage{mathrsfs}
\usepackage{amsfonts}
\usepackage{cases}
\usepackage{latexsym}
\usepackage{amsmath}
\usepackage[arrow,matrix]{xy}
\usepackage{stmaryrd}
\usepackage{amsfonts}
\usepackage{amsmath,amssymb,amscd,bbm,amsthm,mathrsfs,dsfont}

\usepackage{fancyhdr}
\usepackage{amsxtra,ifthen}
\usepackage{verbatim}

\numberwithin{equation}{section}

\theoremstyle{plain}
\newtheorem{theorem}{Theorem}[section]
\newtheorem{lemma}[theorem]{Lemma}
\newtheorem{proposition}[theorem]{Proposition}
\newtheorem{corollary}[theorem]{Corollary}

\theoremstyle{definition}

\newtheorem{question}[theorem]{Question}

\begin{document}

\title[Commuting Traces and Lie Isomorphisms]
{Commuting Traces and Lie Isomorphisms on Generalized Matrix
Algebras}

\author{Zhankui Xiao and Feng Wei}

\address{Xiao: School of Mathematical Science, Huaqiao University,
Quanzhou, Fujian, 362021, P. R. China}

\email{zhkxiao@gmail.com}

\address{Wei: School of Mathematics, Beijing Institute of
Technology, Beijing, 100081, P. R. China}

\email{daoshuowei@gmail.com}\email{daoshuo@hotmail.com}

\begin{abstract}
Let $\mathcal{G}$ be a generalized matrix algebra over a commutative
ring $\mathcal{R}$, ${\mathfrak q}\colon \mathcal{G}\times
\mathcal{G}\longrightarrow \mathcal{G}$ be an $\mathcal{R}$-bilinear
mapping and ${\mathfrak T}_{\mathfrak q}\colon:
\mathcal{G}\longrightarrow \mathcal{G}$ be a trace of
$\mathfrak{q}$. We describe the form of ${\mathfrak T}_{\mathfrak
q}$ satisfying the condition ${\mathfrak T}_{\mathfrak
q}(G)G=G{\mathfrak T}_{\mathfrak q}(G)$ for all $G\in \mathcal{G}$.
The question of when ${\mathfrak T}_{\mathfrak q}$ has the proper
form is considered. Using the aforementioned trace function, we
establish sufficient conditions for each Lie isomorphism of
$\mathcal{G}$ to be almost standard. As applications we characterize
Lie isomorphisms of full matrix algebras, of triangular algebras and
of certain unital algebras with nontrivial idempotents. Some further
research topics related to current work are proposed at the end of
this article.
\end{abstract}

\subjclass[2000]{15A78, 15A86, 16W10}

\keywords{Generalized matrix algebra, commuting trace, Lie
isomorphism}

\thanks{The work of the first author is supported by a research
foundation of Huaqiao University (Grant No. 10BS323).}

\maketitle


\section{Introduction}
\label{xxsec1}

Let $\mathcal{R}$ be a commutative ring with identity, $\mathcal{A}$
be a unital algebra over $\mathcal{R}$ and $\mathcal{Z(A)}$ be the
center of $\mathcal{A}$. Let us denote the commutator or the Lie
product of the elements $a, b\in \mathcal{A}$ by $[a, b]=ab-ba$.
Recall that an $\mathcal{R}$-linear mapping ${\mathfrak f}:
A\longrightarrow \mathcal{A}$ is said to be \textit{commuting} if
$[{\mathfrak f}(a), a]=0$ for all $a\in \mathcal{A}$. When we
investigate a commuting mapping, the principal task is to describe
its form. The identity mapping and every mapping which has its range
in $\mathcal{Z(A)}$ are two classical examples of commuting
mappings. Furthermore, the sum and the pointwise product of
commuting mappings are also commuting mappings. We encourage the
reader to read the well-written survey paper \cite{Bresar3}, in
which the author presented the development of the theory of
commuting mappings and their applications in details.

Let $n$ be a positive integer and $\mathfrak{q}\colon
\mathcal{A}^n\longrightarrow \mathcal{A}$. We say that
$\mathfrak{q}$ is $n$-\textit{linear} if ${\mathfrak q}(a_1,\cdots,
a_n)$ is $\mathcal{R}$-linear in each variable $a_i$, that is,
${\mathfrak q}(a_1, \cdots, ra_i+sb_i, \cdots, a_n)=r {\mathfrak
q}(a_1, \cdots, a_i, \cdots, a_n)+s{\mathfrak q}(a_1, \cdots, b_i,
\cdots, a_n)$ for all $r, s\in \mathcal{R}, a_i, b_i\in \mathcal{A}$
and $i=1, 2, \cdots, n$. The mapping ${\mathfrak T}_{\mathfrak
q}\colon \mathcal{A}\longrightarrow \mathcal{A}$ defined by
${\mathfrak T}_{\mathfrak q}(a)={\mathfrak q}(a, a, \cdots, a)$ is
called a \textit{trace} of ${\mathfrak q}$. We say that a commuting
trace ${\mathfrak T}_{\mathfrak q}$ is \textit{proper} if it is of
the form
$$
{\mathfrak T}_{\mathfrak q}(a)=\sum_{i=0}^n\mu_i(a)a^{n-i},
\hspace{8pt} \forall a\in \mathcal{A},
$$
where $\mu_i(0\leq i\leq n)$ is a mapping from $\mathcal{A}$ into
$\mathcal{Z(A)}$ and each $\mu_i(0\leq i\leq n)$ is in fact a trace
of the $i$-linear mapping ${\mathfrak q}_i$ from $\mathcal{A}^i$
into $\mathcal{Z(A)}$. Let $n=1$ and ${\mathfrak f}\colon
\mathcal{A}\longrightarrow \mathcal{A}$ be an $\mathcal{R}$-linear
mapping. In this case, an arbitrary trace ${\mathfrak T}_{\mathfrak
f}$ of ${\mathfrak f}$ exactly equals to itself. Moreover, if a
commuting trace ${\mathfrak T}_{\mathfrak f}$ of ${\mathfrak f}$ is
proper, then it has the form
$$
{\mathfrak T}_{\mathfrak f}(a)=z a+\eta(a), \hspace{8pt} \forall
a\in \mathcal{A},
$$
where $z\in \mathcal{Z(A)}$ and $\eta$ is an $\mathcal{R}$-linear
mapping from $A$ into $\mathcal{Z(A)}$. Let us see the case of
$n=2$. Suppose that ${\mathfrak g}\colon \mathcal{A}\times
\mathcal{A}\longrightarrow \mathcal{A}$ is an $\mathcal{R}$-bilinear
mapping. If a commuting trace ${\mathfrak T}_{\mathfrak g}$ of
${\mathfrak g}$ is proper, then it is of the form
$$
{\mathfrak T}_{\mathfrak g}(a)=z a^2+\mu(a)a+\nu(a), \hspace{8pt}
\forall a\in \mathcal{A},
$$
where $z\in \mathcal{Z(A)}$, $\mu$ is an $\mathcal{R}$-linear
mapping from $A$ into $\mathcal{Z(A)}$ and $\nu$ is a trace of some
bilinear mapping. It was Bre\v{s}ar who initiated the study of
commuting traces of multilinear mappings in \cite{Bresar1, Bresar2},
where he investigated the structure of commuting traces of
(bi-)linear mappings on prime rings. It has turned out that this
study is closely related to the problem of characterizing Lie
isomorphisms or Lie derivations of associative rings
\cite{BeidarMartindaleMikhalev}. Lee et al further generalized
Bre\v{s}ar's results by showing that each commuting trace of an
arbitrary multilinear mapping on a prime ring has the so-called
proper form \cite{LeeWongLinWang}.

Cheung in \cite{Cheung2} studied commuting mappings of triangular
algebras (e.g., of upper triangular matrix algebras and nest
algebras). He determined the class of triangular algebras for which
every commuting mapping is proper. Xiao and Wei \cite{XiaoWei1}
extended Cheung's result to the generalized matrix algebra case.
They established sufficient conditions for each commuting mapping of
a generalized matrix algebra $
\left[\smallmatrix A & M\\
N & B \endsmallmatrix \right]$ to be proper. Motivated by the
results of Bre\v{s}ar and Cheung, Benkovi\v{c} and Eremita
\cite{BenkovicEremita} considered commuting traces of bilinear
mappings on a triangular algebra $
\left[\smallmatrix A & M\\
O & B \endsmallmatrix \right]$. They gave conditions under which
every commuting trace of a triangular algebra $
\left[\smallmatrix A & M\\
O & B \endsmallmatrix \right]$ is proper. It is worth to mention
that the form of commuting traces of multilinear mappings of upper
triangular matrix algebras was earlier described in
\cite{BeidarBresarChebotar}. One of the main aims of this article is
to provide a sufficient condition for each commuting trace of
arbitrary bilinear mapping on a generalized matrix algebra $
\left[\smallmatrix A & M\\
N & B \endsmallmatrix \right]$ to be proper. Consequently, this make
it possible for us to characterize commuting traces of bilinear
mappings on full matrix algebras, those of bilinear mappings on
triangular algebras and those of bilinear mappings on certain unital
algebras with with a nontrivial idempotent.

Another important purpose of this article is to address the Lie
isomorphisms problem of generalized matrix algebras. At his 1961 AMS
Hour Talk, Herstein proposed many problems concerning the structure
of Jordan and Lie mappings in associative simple and prime rings
\cite{Herstein}. The renowned Herstein's Lie-type mapping research
program was formulated since then. The involved Lie mappings mainly
include Lie isomorphisms, Lie triple isomorphisms, Lie derivations
and Lie triple derivations et al. Given a commutative ring
$\mathcal{R}$ with identity and two associative
$\mathcal{R}$-algebras $\mathcal{A}$ and $\mathcal{B}$, one define a
\textit{Lie isomorphism} from $\mathcal{A}$ into $\mathcal{B}$ to be
an $\mathcal{R}$-linear bijective mapping ${\mathfrak l}$ satisfying
the condition
$$
{\mathfrak l}([a, b])=[{\mathfrak l}(a), {\mathfrak l}(b)],
\hspace{8pt} \forall a, b\in \mathcal{A}.
$$
For example, an isomorphism or a negative of an anti-isomorphism of
one algebra onto another is also a Lie isomorphism. One can ask
whether the converse is true in some special cases. That is, does
every Lie isomorphism between certain associative algebras arise
from isomorphisms and anti-isomorphisms in the sense of modulo
mappings whose range is central ? If $\mathfrak{m}$ is an
isomorphism or the negative of an anti-isomorphism from
$\mathcal{A}$ onto $\mathcal{B}$ and $\mathfrak{n}$ is an
$\mathcal{R}$-linear mapping from $\mathcal{A}$ into the center
$\mathcal{Z(B)}$ of $\mathcal{B}$ such that $\mathfrak{n}([a,b])=0$
for all $a, b\in \mathcal{A}$, then the mapping
$$
\mathfrak{l}=\mathfrak{m}+\mathfrak{n} \eqno(\spadesuit)
$$
is a Lie homomorphism. We shall say that a Lie isomorphism
$\mathfrak{l}\colon A\longrightarrow B$ is \textit{standard} in the
case where it can be expressed in the preceding form $(\spadesuit)$.

The resolution of Herstein's Lie isomorphisms problem in matrix
algebra background has been well-known for a long time. Hua
\cite{Hua} proved that every Lie automorphism of the full matrix
algebra $\mathcal{M}_n(\mathcal{D})(n\geq 3)$ over a division ring
$\mathcal{D}$ is of the standard form $(\spadesuit)$. This result
was extended to the case nonlinear case by Dolinar \cite{Dolinar2}
and \v{S}emrl \cite{Semrl} and was further refined by them.
Dokovi\'{c} \cite{Dokovic} showed that every Lie automorphism of
upper triangular matrix algebras $\mathcal{T}_n(\mathcal{R})$ over a
commutative ring $\mathcal{R}$ without nontrivial idempotents has
the standard form as well. Marcoux and Sourour
\cite{MarcouxSourour1} classified the linear mappings preserving
commutativity in both directions (i.e., $[x,y] = 0$ if and only if
$[\mathfrak{f}(x), \mathfrak{f}(y)]=0$) on upper triangular matrix
algebras $\mathcal{T}_n(\mathbb{F})$ over a field $\mathbb{F}$. Such
a mapping is either the sum of an algebra automorphism of
$\mathcal{T}_n(\mathbb{F})$ (which is inner) and a mapping into the
center $\mathbb{F}I$, or the sum of the negative of an algebra
anti-automorphism and a mapping into the center $\mathbb{F}I$. The
classification of the Lie automorphisms of
$\mathcal{T}_n(\mathbb{F})$ is obtained as a consequence.
Benkovi\v{c} and Eremita \cite{BenkovicEremita} directly applied the
theory of commuting traces to the study of Lie isomorphisms on a
triangular algebra $
\left[\smallmatrix A & M\\
O & B \endsmallmatrix \right]$. They provided sufficient conditions
under which every commuting trace of $
\left[\smallmatrix A & M\\
O & B \endsmallmatrix \right]$ is proper. This is directly applied
to the study of Lie isomorphisms of $
\left[\smallmatrix A & M\\
O & B \endsmallmatrix \right]$. It turns out that under some mild
assumptions, each Lie isomorphism of $
\left[\smallmatrix A & M\\
O & B \endsmallmatrix \right]$ has the standard form $(\spadesuit)$.
On the other hand, Martindale together with some of his students
studied Lie isomorphisms problems of associative rings in a series
of papers \cite{Blau1, Blau2, Martindale1, Martindale2, Martindale4,
Martindale5, Martindale6, Martindale7, Rosen}. Speaking in a loose
manner, the problems have been resolved provided that the rings in
question contain certain nontrivial idempotents. Simultaneously, the
treatment of the problems has been extended from simple rings to
prime rings. The question whether the results on Lie isomorphisms
can be obtained in rings containing no nontrivial idempotents has
been open for a long time. The first idempotent free result on Lie
isomorphisms was obtained in 1993 by Bre\v{s}ar \cite{Bresar1}.
Under some mild technical assumptions (which were removed somewhat
later \cite{BresarSemrl}), he described the form of a Lie
isomorphism between arbitrary prime rings. This was also the first
paper based on applications of the theory of functional identities.
Just recently, Beidar, Bre\v{s}ar, Chebotar, Martindale jointly gave
a final solution to the long-standing Herstein's conjecture of Lie
isomorphisms of prime rings using the theory of functional
identities, see the paper \cite{BeidarBresarChebotarMartindale} and
references therein. Simultaneously, Lie isomorphisms between rings
and between (non-)self-adjoint operator algebras have received a
fair amount of attention and have also been intensively studied. The
involved rings and operator algebras include (semi-)prime rings, the
algebra of bounded linear operators, $C^\ast$-algebras, von Neumann
algebras, $H^\ast$-algebras, Banach space nest algebras, Hilbert
space nest algebras, reflexive algebras, see \cite{BaiDuHou1,
BaiDuHou2, BanningMathieu, CalderonGonzalez1, CalderonGonzalez2,
CalderonGonzalez3, CalderonHaralampidou, Lu, MarcouxSourour2,
Mathieu, Miers1, Miers2, QiHou1, QiHou2, Semrl, Sourour, WangLu,
YuLu, ZhangZhang}.

This is the first paper in a series of two that we are planning on
this topic. The second paper will be dedicated to studying, in more
detail, centralizing traces and Lie triple isomorphisms on
triangular algebras and those mappings on generalized matrix
algebras \cite{XiaoWei2}. The roadmap of this paper is as follows.
Section $2$ contains the definition of generalized matrix algebra
and some classical examples. In Section $3$ we provide sufficient
conditions for each commuting trace of arbitrary bilinear mapping on
a generalized matrix algebra $
\left[\smallmatrix A & M\\
N & B \endsmallmatrix \right]$ to be proper (Theorem
\ref{xxsec3.4}). And then we apply this result to describe the
commuting traces of various generalized matrix algebras. In Section
$4$ we will give sufficient conditions under which every Lie
isomorphism from a generalized matrix algebra into another one has
the standard form (Theorem \ref{xxsec4.3}). As corollaries of
Theorem \ref{xxsec4.3}, characterizations of Lie isomorphisms on
triangular algebras, on full matrix algebras and on certain unital
algebras with nontrivial idempotents are obtained. The last section
contains some potential future research topics related to our
current work.

\section{Generalized Matrix Algebras and Examples}\label{xxsec2}

Let us begin with the definition of generalized matrix algebras
given by a Morita context. Let $\mathcal{R}$ be a commutative ring
with identity. A \textit{Morita context} consists of two
$\mathcal{R}$-algebras $A$ and $B$, two bimodules $_AM_B$ and
$_BN_A$, and two bimodule homomorphisms called the pairings
$\Phi_{MN}: M\underset {B}{\otimes} N\longrightarrow A$ and
$\Psi_{NM}: N\underset {A}{\otimes} M\longrightarrow B$ satisfying
the following commutative diagrams:
$$
\xymatrix{ M \underset {B}{\otimes} N \underset{A}{\otimes} M
\ar[rr]^{\hspace{8pt}\Phi_{MN} \otimes I_M} \ar[dd]^{I_M \otimes
\Psi_{NM}} && A
\underset{A}{\otimes} M \ar[dd]^{\cong} \\  &&\\
M \underset{B}{\otimes} B \ar[rr]^{\hspace{10pt}\cong} && M }
\hspace{6pt}{\rm and}\hspace{6pt} \xymatrix{ N \underset
{A}{\otimes} M \underset{B}{\otimes} N
\ar[rr]^{\hspace{8pt}\Psi_{NM}\otimes I_N} \ar[dd]^{I_N\otimes
\Phi_{MN}} && B
\underset{B}{\otimes} N \ar[dd]^{\cong}\\  &&\\
N \underset{A}{\otimes} A \ar[rr]^{\hspace{10pt}\cong} && N
\hspace{2pt}.}
$$
Let us write this Morita context as $(A, B, M, N, \Phi_{MN},
\Psi_{NM})$. We refer the reader to \cite{Morita} for the basic
properties of Morita contexts. If $(A, B, M, N,$ $ \Phi_{MN},
\Psi_{NM})$ is a Morita context, then the set
$$
\left[
\begin{array}
[c]{cc}%
A & M\\
N & B\\
\end{array}
\right]=\left\{ \left[
\begin{array}
[c]{cc}%
a& m\\
n & b\\
\end{array}
\right] \vline a\in A, m\in M, n\in N, b\in B \right\}
$$
form an $\mathcal{R}$-algebra under matrix-like addition and
matrix-like multiplication, where at least one of the two bimodules
$M$ and $N$ is distinct from zero. Such an $\mathcal{R}$-algebra is
usually called a \textit{generalized matrix algebra} of order $2$
and is denoted by
$$
\mathcal{G}=\left[
\begin{array}
[c]{cc}%
A & M\\
N & B\\
\end{array}
\right].
$$
In a similar way, one can define a generalized matrix algebra of
order $n>2$. It was shown that up to isomorphism, arbitrary
generalized matrix algebra of order $n$ $(n\geq 2)$ is a generalized
matrix algebra of order 2 \cite[Example 2.2]{LiWei}. If one of the
modules $M$ and $N$ is zero, then $\mathcal{G}$ exactly degenerates
to an \textit{upper triangular algebra} or a \textit{lower
triangular algebra}. In this case, we denote the resulted upper
triangular algebra (resp. lower triangular algebra) by
$$\mathcal{T^U}=
\left[
\begin{array}
[c]{cc}%
A & M\\
O & B\\
\end{array}
\right]   \hspace{8pt} \left({\rm resp.} \hspace{4pt} \mathcal{T_L}=
\left[
\begin{array}
[c]{cc}%
A & O\\
N & B\\
\end{array}
\right]\right)
$$
Note that our current generalized matrix algebras contain those
generalized matrix algebras in the sense of Brown \cite{Brown} as
special cases. Let $\mathcal{M}_n(\mathcal{R})$ be the full matrix
algebra consisting of all $n\times n$ matrices over $\mathcal{R}$.
It is worth to point out that the notion of generalized matrix
algebras efficiently unifies triangular algebras with full matrix
algebras together. The distinguished feature of our systematic work
is to deal with all questions related to (non-)linear mappings of
triangular algebras and of full matrix algebras under a unified
frame, which is the admired generalized matrix algebras frame, see
\cite{LiWei, LiWykWei, XiaoWei1}.

Let us list some classical examples of generalized matrix algebras
which will be revisited in the sequel (Section \ref{xxsec3} and
Section \ref{xxsec4}). Since these examples have already been
presented in many papers, we just state their title without any
introduction. We refer the reader to \cite{LiWei, XiaoWei1} for more
details.
\begin{enumerate}
\item[(\rm a)] Unital algebras with nontrivial
idempotents;
\item[(\rm b)] Full matrix algebras;
\item[(\rm c)] Inflated algebras;
\item[(\rm d)] Upper and lower triangular matrix algebras;
\item[(\rm e)] Hilbert space nest algebras
\end{enumerate}

\section{Commuting Traces of Bilinear Mappings on Generalized
Matrix Algebras}\label{xxsec3}

In this section we will establish sufficient conditions for each
commuting trace of an arbitrary bilinear mapping on a generalized
matrix algebra $
\left[\smallmatrix A & M\\
N & B \endsmallmatrix \right]$ to be proper (Theorem
\ref{xxsec3.4}). Consequently, we are able to describe commuting
traces of bilinear mappings on triangular algebras, on full matrix
algebras and on certain unital algebras with nontrivial idempotents.
The most important is that Theorem \ref{xxsec3.4} will be used to
characterize Lie isomorphisms from a generalized matrix algebras
into another in Section \ref{xxsec4}. In addition, Beidar,
Bre\v{s}ar and Chebotar in \cite{BeidarBresarChebotar} described the
form of commuting traces of multilinear mappings on upper triangular
matrix algebras. Motivated by their joint work, we propose a
conjecture concerning commuting traces of multilinear mappings on
generalized matrix algebras.

Throughout this section, we denote the generalized matrix algebra of
order $2$ originated from the Morita context $(A, B, _AM_B, _BN_A,
\Phi_{MN}, \Psi_{NM})$ by
$$
\mathcal{G}=\left[
\begin{array}
[c]{cc}%
A & M\\
N & B
\end{array}
\right] ,
$$
where at least one of the two bimodules $M$ and $N$ is distinct from
zero. We always assume that $M$ is faithful as a left $A$-module and
also as a right $B$-module, but no any constraint conditions on $N$.
The center of $\mathcal{G}$ is
$$
\mathcal{Z(G)}=\left\{ \left[
\begin{array}
[c]{cc}%
a & 0\\
0 & b
\end{array}
\right] \vline \hspace{3pt} am=mb, \hspace{3pt} na=bn,\ \forall\
m\in M, \hspace{3pt} \forall n\in N \right\}.
$$
Indeed, by \cite[Lemma 1]{Krylov} we know that the center
$\mathcal{Z(G)}$ consists of all diagonal matrices $
\left[\smallmatrix a & 0\\
0 & b
\endsmallmatrix \right]$,
where $a\in \mathcal{Z}(A)$, $b\in \mathcal{Z}(B)$ and $am=mb$,
$na=bn$ for all $m\in M, n\in N$. However, in our situation which
$M$ is faithful as a left $A$-module and also as a right $B$-module,
the conditions that $a\in \mathcal{Z}(A)$ and $b\in \mathcal{Z}(B)$
become redundant and can be deleted. Indeed, if $am=mb$ for all
$m\in M$, then for any $a^\prime \in A$ we get
$$
(aa^\prime-a^\prime a)m=a(a'm)-a'(am)=(a'm)b-a'(mb)=0.
$$
The assumption that $M$ is faithful as a left $A$-module leads to
$aa'-a'a=0$ and hence $a\in \mathcal{Z}(A)$. Likewise, we also have
$b\in \mathcal{Z}(B)$.

Let us define two natural $\mathcal{R}$-linear projections
$\pi_A:\mathcal{G}\rightarrow A$ and $\pi_B:\mathcal{G}\rightarrow
B$ by
$$
\pi_A: \left[
\begin{array}
[c]{cc}%
a & m\\
n & b\\
\end{array}
\right] \longmapsto a \quad \text{and} \quad \pi_B: \left[
\begin{array}
[c]{cc}%
a & m\\
n & b\\
\end{array}
\right] \longmapsto b.
$$
By the above paragraph, it is not difficult to see that $\pi_A
\left(\mathcal{Z(G)}\right)$ is a subalgebra of $\mathcal{Z}(A)$ and
that $\pi_B\left(\mathcal{Z(G)}\right)$ is a subalgebra of
$\mathcal{Z}(B)$.
Given an element $a\in\pi_A(\mathcal{Z(G)})$, if $\left[\smallmatrix a & 0\\
0 & b
\endsmallmatrix \right], \left[\smallmatrix a & 0\\
0 & b^\prime
\endsmallmatrix \right] \in \mathcal{Z(G)}$, then we have $am=mb=mb'$ for
all $m\in M$. Since $M$ is faithful as a right $B$-module,
$b=b^\prime$. That implies there exists a unique
$b\in\pi_B(\mathcal{Z(G)})$, which is denoted by $\varphi(a)$, such
that $
\left[\smallmatrix a & 0\\
0 & b
\endsmallmatrix \right] \in \mathcal{Z(G)}$. It is easy to
verify that the map $\varphi:\pi_A(\mathcal{Z(G)})\longrightarrow
\pi_B(\mathcal{Z(G)})$ is an algebraic isomorphism such that
$am=m\varphi(a)$ and $na=\varphi(a)n$ for all $a\in
\pi_A(\mathcal{Z(G)}), m\in M, n\in N$.

Let $\mathcal{A}$ and $\mathcal{B}$ be algebras. Recall an
$(\mathcal{A}, \mathcal{B})$-bimodule is \textit{loyal} if
$a\mathcal{M}b=0$ implies that $a=0$ or $b=0$ for all $a\in
\mathcal{A}, b\in \mathcal{B}$. Let us first state several lemmas
without proofs, since their proofs are identical with those of
\cite[Lemma 2.4, Lemma 2.5, Lemma 2.6]{BenkovicEremita}.

\begin{lemma}\label{xxsec3.1}
Let $\mathcal{G}=\left[
\begin{array}
[c]{cc}%
A & M\\
N & B\\
\end{array}
\right]$ be a generalized matrix algebra with a loyal $(A,
B)$-bimodule $M$. For arbitrary element $\lambda\in
\pi_A(\mathcal{Z(G)})$ and arbitrary nonzero element $a\in A$, if
$\lambda a=0$, then $\lambda=0$
\end{lemma}

\begin{lemma}\label{xxsec3.2}
Let $\mathcal{G}=\left[
\begin{array}
[c]{cc}%
A & M\\
N & B\\
\end{array}
\right]$ be a generalized matrix algebra with a loyal
$(A,B)$-bimodule $M$. Then the center $\mathcal{Z(G)}$ of
$\mathcal{G}$ is a domain.
\end{lemma}

\begin{lemma}\label{xxsec3.3}
The generalized matrix algebra $\mathcal{G}=\left[
\begin{array}
[c]{cc}%
A & M\\
N & B\\
\end{array}
\right]$ has no nonzero central ideals.
\end{lemma}

We are ready to state and prove the main result of this section.

\begin{theorem}\label{xxsec3.4}
Let $\mathcal{G}=\left[
\begin{array}
[c]{cc}%
A & M\\
N & B\\
\end{array}
\right]$  be a $2$-torsionfree generalized matrix algebra over a
commutative ring $\mathcal{R}$ and ${\mathfrak q}\colon
\mathcal{G}\times \mathcal{G}\longrightarrow \mathcal{G}$ be an
$\mathcal{R}$-bilinear mapping. If
\begin{enumerate}
\item[{\rm(1)}] every commuting linear mapping on $A$ or $B$ is proper;
\item[{\rm(2)}] $\pi_A(\mathcal{Z(G)})=\mathcal{Z}(A)\neq A$ and
$\pi_B(\mathcal{Z(G)})=\mathcal{Z}(B)\neq
B$;
\item[{\rm(3)}] $M$ is loyal,
\end{enumerate}
then every commuting trace ${\mathfrak T}_{\mathfrak q}:
\mathcal{G}\longrightarrow \mathcal{G}$ of ${\mathfrak q}$ is
proper.
\end{theorem}

For convenience, let us write $A_1=A$, $A_2=M$, $A_3=N$ and $A_4=B$.
Suppose that ${\mathfrak T}_{\mathfrak q}$ is an arbitrary trace of
the $\mathcal{R}$-bilinear mapping ${\mathfrak q}$. Then there exist
$\mathcal{R}$-bilinear mappings $f_{ij}: A_i\times A_j\rightarrow
A_1$, $g_{ij}: A_i\times A_j\rightarrow A_2$, $h_{ij}: A_i\times
A_j\rightarrow A_3$ and $k_{ij}: A_i\times A_j\rightarrow A_4$,
$1\leqslant i\leqslant j\leqslant 4$, such that
$$\begin{aligned}
{\mathfrak T}_{\mathfrak q}: \mathcal{G}& \longrightarrow \mathcal{G}\\
\left[
\begin{array}
[c]{cc}%
a_1 & a_2\\
a_3 & a_4\\
\end{array}
\right]& \longmapsto \left[
\begin{array}
[c]{cc}%
F(a_1,a_2,a_3,a_4) & G(a_1,a_2,a_3,a_4)\\
H(a_1,a_2,a_3,a_4) & K(a_1,a_2,a_3,a_4)\\
\end{array}
\right], \forall \left[\begin{array}
[c]{cc}%
a_1 & a_2\\
a_3 & a_4\\
\end{array}\right]\in \mathcal{G}
\end{aligned}
$$
where
$$
F(a_1,a_2,a_3,a_4)=\sum_{1\leqslant i\leqslant j\leqslant
4}f_{ij}(a_i,a_j),
$$
$$ G(a_1,a_2,a_3,a_4)=\sum_{1\leqslant i\leqslant
j\leqslant 4}g_{ij}(a_i,a_j),
$$
$$
H(a_1,a_2,a_3,a_4)=\sum_{1\leqslant i\leqslant j\leqslant
4}h_{ij}(a_i,a_j),
$$
$$ K(a_1,a_2,a_3,a_4)=\sum_{1\leqslant i\leqslant
j\leqslant 4}k_{ij}(a_i,a_j).
$$
Since ${\mathfrak T}_{\mathfrak q}$ is commuting, we have
$$\begin{aligned}
0&=\left[\left[
\begin{array}
[c]{cc}%
F & G\\
H & K\\
\end{array}
\right],\left[
\begin{array}
[c]{cc}%
a_1 & a_2\\
a_3 & a_4\\
\end{array}
\right]\right]\\
&=\left[
\begin{array}
[c]{cc}%
Fa_1+Ga_3-a_1F-a_2H & Fa_2+Ga_4-a_1G-a_2K\\
Ha_1+Ka_3-a_3F-a_4H & Ha_2+Ka_4-a_3G-a_4K\\
\end{array}
\right]
\end{aligned} \eqno(\bigstar)
$$
for all $\left[
\begin{array}
[c]{cc}%
a_1 & a_2\\
a_3 & a_4\\
\end{array}
\right]\in \mathcal{G}$.

Now we divide the proof of Theorem \ref{xxsec3.4} into a series of
lemmas for comfortable reading.

\begin{lemma}\label{xxsec3.5}
$H(a_1,a_2,a_3,a_4)=h_{13}(a_1,a_3)+h_{23}(a_2,a_3)+h_{33}(a_3,a_3)+h_{34}(a_3,a_4)$.
\end{lemma}

\begin{proof}
It follows from the matrix relation $(\bigstar)$ that
$$
Ha_1+Ka_3-a_3F-a_4H=0. \eqno(3.1)
$$
Let us choose $a_2=0$, $a_3=0$ and $a_4=0$. Then $(3.1)$ implies
that $h_{11}(a_1,a_1)a_1=0$ for all $a_1\in A_1$. Obviously,
$h_{11}(1,1)=0$. Replacing $a_1$ by $a_1+1$ and $1-a_1$ in
$h_{11}(a_1,a_1)a_1=0$ in turn, we obtain
$$
(h_{11}(a_1,a_1)+h_{11}(a_1,1)+h_{11}(1,a_1))(a_1+1)=0
$$
and
$$
(h_{11}(a_1,a_1)-h_{11}(a_1,1)-h_{11}(1,a_1))(1-a_1)=0
$$
for all $a_1\in A_1$. Combining the above two equations yields that
$2(h_{11}(a_1,1)+h_{11}(1,a_1))=0$. Since $\mathcal{G}$ is
$2$-torsion free, $h_{11}(a_1,a_1)=0$ for all $a_1\in A_1$.

Let us take $a_3=0$ and $a_4=0$ in $(3.1)$. Then we get
$$
(h_{12}(a_1,a_2)+h_{22}(a_2,a_2))a_1=0 \eqno(3.2)
$$
for all $a_1\in A_1, a_2\in A_2$. Substituting $-a_2$ for $a_2$ in
$(3.2)$ gives
$$
(-h_{12}(a_1,a_2)+h_{22}(a_2,a_2))a_1=0 \eqno(3.3)
$$
for all $a_1\in A_1, a_2\in A_2$. By $(3.2)$ and $(3.3)$ we know
that $2h_{22}(a_2,a_2)a_1=0$ for all $a_1\in A_1, a_2\in A_2$. Hence
$h_{22}(a_2,a_2)=0$ for all $a_2\in A_2$.

Now the relation $(3.2)$ shows that $h_{12}(a_1,a_2)a_1=0$ for all
$a_1\in A_1, a_2\in A_2$. Thus $h_{12}(1,a_2)=0$. Replacing $a_1$ by
$a_1+1$ in $h_{12}(a_1,a_2)a_1=0$ leads to
$0=(h_{12}(a_1,a_2)+h_{12}(1,a_2))(a_1+1)=h_{12}(a_1,a_2)$.

Let us choose $a_1=0$, $a_2=0$ and $a_3=0$. Applying $(3.1)$ yields
that $a_4h_{44}(a_4,a_4)$ $=0$ for all $a_4\in A_4$. Therefore
$h_{44}(1,1)=0$. Substituting $a_4+1$ and $1-a_4$ for $a_4$ in
$a_4h_{44}(a_4,a_4)=0$ in turn, we arrive at
$$
(a_4+1)(h_{44}(a_4,a_4)+h_{44}(a_4,1)+h_{44}(1,a_4))=0
$$
and
$$
(1-a_4)(h_{44}(a_4,a_4)-h_{44}(a_4,1)-h_{44}(1,a_4))=0
$$
for all $a_4\in A_4$. Combining the above two equations gives
$2(h_{44}(a_4,1)+h_{44}(1,a_4))=0$. Since $\mathcal{G}$ is
$2$-torsion free, $h_{44}(a_4,a_4)=0$ for all $a_4\in A_4$.

If we take $a_1=0$ and $a_3=0$ into $(3.1)$, then
$$
a_4(h_{22}(a_2,a_2)+h_{24}(a_2,a_4))=0 \eqno(3.4)
$$
for all $a_2\in A_2, a_4\in A_4$. Note that the fact
$h_{22}(a_2,a_2)=0$ for all $a_2\in A_2$. Hence $(3.4)$ implies that
$a_4h_{24}(a_2,a_4)=0$ for all $a_2\in A_2, a_4\in A_4$. So
$h_{24}(a_2,1)=0$. Replacing $a_4$ by $a_4+1$ in
$a_4h_{24}(a_2,a_4)=0$, we obtain
$0=(a_4+1)(h_{24}(a_2,a_4)+h_{24}(a_2,1))=h_{24}(a_2,a_4)$.

Finally let us choose $a_3=0$. Then $(3.1)$ becomes
$$
h_{14}(a_1,a_4)a_1-a_4h_{14}(a_1,a_4)=0 \eqno(3.5)
$$
for all $a_1\in A_1, a_4\in A_4$. Replacing $a_1$ by $-a_1$ in
$(3.5)$ we have
$$
h_{14}(a_1,a_4)a_1+a_4h_{14}(a_1,a_4)=0 \eqno(3.6)
$$
for all $a_1\in A_1, a_2\in A_2$. Combining $(3.5)$ with $(3.6)$
yields $h_{14}(a_1,a_4)a_1=0$ for all $a_1\in A_1, a_4\in A_4$.
Clearly, $h_{14}(1, a_4)=0$ for all $a_4\in A_4$. Substituting
$a_1+1$ for $a_1$ in $h_{14}(a_1,a_4)a_1=0$, we get
$0=(h_{14}(a_1,a_4)+h_{14}(1,a_4))(a_1+1)=h_{14}(a_1,a_4)$ and this
completes the proof of the lemma.
\end{proof}

Similarly, we can show

\begin{lemma}\label{xxsec3.6}
$G(a_1,a_2,a_3,a_4)=g_{12}(a_1,a_2)+g_{22}(a_2,a_2)+g_{23}(a_2,a_3)+g_{24}(a_2,a_4)$.
\end{lemma}

\begin{lemma}\label{xxsec3.7}
With notations as above, we have
\begin{enumerate}
\item[{\rm(1)}] $a_1\mapsto f_{11}(a_1,a_1)$ is a commuting trace;

\item[{\rm(2)}] $a_1\mapsto f_{12}(a_1,a_2)$, $a_1\mapsto
f_{13}(a_1,a_3)$, $a_1\mapsto f_{14}(a_1,a_4)$ are commuting linear
mappings for each $a_2\in A_2, a_3\in A_3, a_4\in A_4$,
respectively;

\item[{\rm(3)}] $f_{22},f_{24},f_{33},f_{34},f_{44}$
map into $\mathcal{Z}(A_1)$.
\end{enumerate}
\end{lemma}

\begin{proof}
It follows from the matrix relation $(\bigstar)$ that
$$
Fa_1+Ga_3-a_1F-a_2H=0. \eqno(3.7)
$$
Let us take $a_2=0$, $a_3=0$ and $a_4=0$ in $(3.7)$. Thus
$[f_{11}(a_1,a_1),a_1]=0$ for all $a_1\in A_1$.

Let us choose $a_3=0$ and $a_4=0$. Applying Lemma \ref{xxsec3.5} and
$(3.7)$ yields $[F,a_1]=0$, that is
$$
[f_{12}(a_1,a_2)+f_{22}(a_2,a_2),a_1]=0 \eqno(3.8)
$$
for all $a_1\in A_1, a_2\in A_2$. Replacing $a_1$ by $-a_1$ in
$(3.8)$ we obtain
$$
[f_{12}(a_1,a_2)-f_{22}(a_2,a_2),a_1]=0 \eqno(3.9)
$$
for all $a_1\in A_1, a_2\in A_2$. Combining $(3.8)$ with $(3.9)$ we
get $[f_{12}(a_1,a_2),a_1]=0$ and $[f_{22}(a_2,a_2),a_1]=0$ for all
$a_1\in A_1, a_2\in A_2$.

If we take $a_3=0$, then $(3.7)$ and Lemma \ref{xxsec3.5} imply that
$$
[f_{14}(a_1,a_4)+f_{24}(a_2,a_4)+f_{44}(a_4,a_4),a_1]=0 \eqno(3.10)
$$
for all $a_1\in A_1, a_2\in A_2, a_4\in A_4$. Substituting $-a_1$
for $a_1$ in $(3.10)$ we have
$$
[f_{14}(a_1,a_4)-f_{24}(a_2,a_4)-f_{44}(a_4,a_4),a_1]=0 \eqno(3.11)
$$
for all $a_1\in A_1, a_2\in A_2, a_4\in A_4$. In view of $(3.10)$
and $(3.11)$, we arrive at $[f_{14}(a_1,a_4), $ $a_1]=0$ and
$[f_{24}(a_2,a_4)+f_{44}(a_4,a_4),a_1]=0$. Taking $a_2=0$ into the
last equality we get $f_{44}(a_4,a_4)\in \mathcal{Z}(A_1)$ and hence
$f_{24}(a_2,a_4)\in \mathcal{Z}(A_1)$ for all $a_2\in A_2, a_4\in
A_4$.

Let us choose $a_2=0$. By $(3.7)$ and Lemma \ref{xxsec3.6} it
follows that
$$
[f_{13}(a_1,a_3)+f_{33}(a_3,a_3)+f_{34}(a_3,a_4),a_1]=0 \eqno(3.12)
$$
for all $a_1\in A_1, a_3\in A_3, a_4\in A_4$. Let us put $a_4=0$ in
$(3.12)$. Then
$$
[f_{13}(a_1,a_3)+f_{33}(a_3,a_3),a_1]=0 \eqno(3.13)
$$
for all $a_1\in A_1, a_3\in A_3$, which gives $f_{34}(a_3,a_4)\in
\mathcal{Z}(A_1)$. Replacing $a_1$ by $-a_1$ in $(3.13)$ yields
$$
[f_{13}(a_1,a_3)-f_{33}(a_3,a_3),a_1]=0 \eqno(3.14)
$$
for all $a_1\in A_1, a_3\in A_3$. Combining $(3.13)$ with $(3.14)$
we obtain $f_{33}(a_3,a_3)\in \mathcal{Z}(A_1)$ and
$[f_{13}(a_1,a_3),a_1]=0$ for all $a_1\in A_1, a_3\in A_3$.
\end{proof}

Using an analogous proof of Lemma \ref{xxsec3.7} the following
results hold.

\begin{lemma}\label{xxsec3.8}
With notations as above, we have
\begin{enumerate}
\item[{\rm(1)}] $a_4\mapsto k_{44}(a_4,a_4)$ is a commuting trace;

\item[{\rm(2)}] $a_4\mapsto k_{14}(a_1,a_4)$, $a_4\mapsto k_{24}(a_2,a_4)$,
$a_4\mapsto k_{34}(a_3,a_4)$ are commuting mappings for each $a_1\in
A_1, a_2\in A_2, a_3\in A_3$, respectively;

\item[{\rm(3)}] $k_{11},k_{12},k_{13},k_{22},k_{33}$ map into $\mathcal{Z}(A_4)$.
\end{enumerate}
\end{lemma}

\begin{lemma}\label{xxsec3.9}
$\left[
\begin{array}
[c]{cc}%
f_{22}(a_2,a_2) & 0\\
0 & k_{22}(a_2,a_2)
\end{array}
\right] \in \mathcal{Z}(\mathcal{G})$ and $\left[
\begin{array}
[c]{cc}%
f_{33}(a_3,a_3) & 0\\
0 & k_{33}(a_3,a_3)
\end{array}
\right] \in \mathcal{Z}(\mathcal{G})$.
\end{lemma}

\begin{proof}
By the relation $(\bigstar)$ we know that
$$
Fa_2+Ga_4-a_1G-a_2K=0. \eqno(3.15)
$$
Let us take $a_1=0$ and $a_4=0$. Then $(3.15)$ implies that
$$
(f_{22}(a_2,a_2)+f_{23}(a_2,a_3)+f_{33}(a_3,a_3))a_2=a_2(k_{22}(a_2,a_2)+k_{23}(a_2,a_3)+k_{33}(a_3,a_3))\eqno(3.16)
$$
for all $a_2\in A_2, a_3\in A_3$. Moreover, setting $a_3=0$ in
$(3.16)$ we get
$$f_{22}(a_2,a_2)a_2=a_2k_{22}(a_2,a_2) \eqno(3.17)
$$
for all $a_2\in A_2$. Applying Lemma \ref{xxsec3.7}, Lemma
\ref{xxsec3.8} and \cite[Lemma 3.2]{XiaoWei1} yields that
$(f_{22}(a_2,a_2)-\varphi^{-1}(k_{22}(a_2,a_2)))a_2=0$. By the
complete linearization we have
$$
\beta(x,y)z+\beta(z,x)y+\beta(y,z)x=0 \eqno(3.18)
$$
for all $x,y,z\in A_2$, where
$$
\beta(x,y)=f_{22}(x,y)-\varphi^{-1}(k_{22}(x,y))+f_{22}(y,x)-\varphi^{-1}(k_{22}(y,x)).
$$
Obviously, the mapping $\beta: A_2\times A_2\rightarrow
\mathcal{Z}(A_1)$ is bilinear and symmetric. By the hypothesis there
exist $a,b\in A_1$ such that $[a,b]\neq 0$. Replacing $z$ by $az$ in
$(3.18)$ and subtracting the left multiplication of $(3.18)$ by $a$,
we get
$$
(\beta(az,x)-\beta(z,x)a)y+(\beta(y,az)-\beta(y,z)a)x=0
$$
for all $x,y,z\in A_2$. It follows from \cite[Lemma
2.3]{BenkovicEremita} that $\beta(az,x)=\beta(z,x)a$ and hence
$\beta(z,x)[a,b]=0$ for all $x,z\in A_2$. Applying Lemma
\ref{xxsec3.1} yields $\beta=0$. In particular, $\beta(a_2,a_2)=0$
for all $a_2\in A_2$. Thus $\left[
\begin{array}
[c]{cc}%
f_{22}(a_2,a_2) & 0\\
0 & k_{22}(a_2,a_2)
\end{array}
\right]\in \mathcal{Z(G)}$.

Now the relation $(3.16)$ becomes
$$
(f_{23}(a_2,a_3)+f_{33}(a_3,a_3))a_2=a_2(k_{23}(a_2,a_3)+k_{33}(a_3,a_3))\eqno(3.19)
$$
for all $a_2\in A_2, a_3\in A_3$. Substituting $-a_2$ for $a_2$ and
applying $(3.19)$, we arrive at
$f_{33}(a_3,a_3)a_2=a_2k_{33}(a_3,a_3)$ for all $a_2\in A_2, a_3\in
A_3$. In view of the fact $M$ is faithful as a left $A$-module and
$k_{33}(a_3,a_3)\in \mathcal{Z}(B)=\pi_B(\mathcal{Z}(\mathcal{G}))$,
we assert that $\left[
\begin{array}
[c]{cc}%
f_{33}(a_3,a_3) & 0\\
0 & k_{33}(a_3,a_3)
\end{array}
\right]\in \mathcal{Z(G)}$.
\end{proof}

\begin{lemma}\label{xxsec3.10}
$f_{12}(a_1,a_2)=\alpha(a_2)a_1+\varphi^{-1}(k_{12}(a_1,a_2))$ and
$k_{24}(a_2,a_4)=\varphi(\alpha(a_2))a_4$ $+\varphi
(f_{24}(a_2,a_4))$, where
$\alpha(a_2)=f_{12}(1,a_2)-\varphi^{-1}(k_{12}(1,a_2))$.
\end{lemma}

\begin{proof}
Taking $a_4=0$ into $(3.15)$ and using $(3.16)$ we have
$$
\big(f_{11}(a_1,a_1)+f_{12}(a_1,a_2)+f_{13}(a_1,a_3)\big)a_2-a_2\big(k_{11}(a_1,a_1)+k_{12}(a_1,a_2)
+k_{13}(a_1,a_3)\big)$$
$$
-a_1\big(g_{12}(a_1,a_2)+g_{22}(a_2,a_2)+g_{23}(a_2,a_3)\big)=0.
\eqno(3.20)
$$
Replacing $a_1$ by $-a_1$ in $(3.20)$ we get
$$
\big(f_{11}(a_1,a_1)-f_{12}(a_1,a_2)-f_{13}(a_1,a_3)\big)a_2-a_2\big(k_{11}(a_1,a_1)-k_{12}(a_1,a_2)
-k_{13}(a_1,a_3)\big)$$
$$
-a_1\big(g_{12}(a_1,a_2)-g_{22}(a_2,a_2)-g_{23}(a_2,a_3)\big)=0.
\eqno(3.21)
$$
Combining $(3.20)$ with $(3.21)$ yields
$$
a_1g_{12}(a_1,a_2)=f_{11}(a_1,a_1)a_2-a_2k_{11}(a_1,a_1),
\eqno(3.22)
$$
$$
a_1g_{22}(a_2,a_2)=f_{12}(a_1,a_2)a_2-a_2k_{12}(a_1,a_2),
\eqno(3.23)
$$
$$
a_1g_{23}(a_2,a_3)=f_{13}(a_1,a_3)a_2-a_2k_{13}(a_1,a_3).
\eqno(3.24)
$$
In an analogous way, taking $a_1=0$ into $(3.15)$ and using $(3.16)$
we obtain
$$
g_{24}(a_2,a_4)a_4=a_2k_{44}(a_4,a_4)-f_{44}(a_4,a_4)a_2,
\eqno(3.25)
$$
$$
g_{22}(a_2,a_2)a_4=a_2k_{24}(a_2,a_4)-f_{24}(a_2,a_4)a_2,
\eqno(3.26)
$$
$$
g_{23}(a_2,a_3)a_4=a_2k_{34}(a_3,a_4)-f_{34}(a_3,a_4)a_2.
\eqno(3.27)
$$
On the other hand, we have showed that $[f_{12}(a_1,a_2),a_1]=0$ for
all $a_1\in A_1, a_2\in A_2$. Substituting $a_1+1$ for $a_1$ leads
to $f_{12}(1,a_2)\in \mathcal{Z}(A_1)$ for all $a_2\in A_2$. By the
relation $(3.23)$ we know that
$$
g_{22}(a_2,a_2)=\alpha(a_2)a_2, \eqno(3.28)
$$
where $\alpha(a_2)=f_{12}(1,a_2)-\varphi^{-1}(k_{12}(1,a_2))\in
\mathcal{Z}(A_1)$. Let us set
$E(a_1,a_2)=f_{12}(a_1,a_2)-\alpha(a_2)a_1-\varphi^{-1}(k_{12}(a_1,a_2))$.
Then $(3.23)$ and $(3.28)$ jointly imply that $E(a_1,a_2)a_2$ $=0$,
which further gives $E(a_1,a_2)b_2+E(a_1,b_2)a_2=0$ for all $a_1\in
A_1$ and $a_2,b_2\in A_2$. By \cite[Lemma 2.3]{BenkovicEremita} we
conclude that $E(a_1, a_2)=0$. Hence
$f_{12}(a_1,a_2)=\alpha(a_2)a_1+\varphi^{-1}(k_{12}(a_1,a_2))$.
Similarly, we can show that $k_{24}$ is of the desired form as well.
\end{proof}

\begin{lemma}\label{xxsec3.11}
$f_{13}(a_1,a_3)=\tau(a_3)a_1+\varphi^{-1}(k_{13}(a_1,a_3))$ and
$k_{34}(a_3,a_4)=\varphi(\tau(a_3))a_4$ $+\varphi
(f_{34}(a_3,a_4))$, where
$\tau(a_3)=f_{13}(1,a_3)-\varphi^{-1}(k_{13}(1,a_3))$.
\end{lemma}

\begin{proof}
Note that $[f_{13}(a_1,a_3),a_1]=0$ for all $a_1\in A_1, a_3\in
A_3$. Substituting $a_1+1$ for $a_1$ gives $f_{13}(1,a_3)\in
\mathcal{Z}(A_1)$ for all $a_3\in A_3$. Let us set
$\tau(a_3)=f_{13}(1,a_3)-\varphi^{-1}(k_{13}(1,a_3))$ and
$E(a_1,a_3)=f_{13}(a_1,a_3)-\tau(a_3)a_1-\varphi^{-1}(k_{13}(a_1,a_3))$.
It follows from $(3.24)$ that $E(a_1,a_3)a_2=0$ for all $a_1\in A_1,
a_2\in A_2, a_3\in A_3$. Since $M=A_2$ is faithful as a left
$A$-module, we obtain $E(a_1, a_3)=0$ and hence
$f_{13}(a_1,a_3)=\tau(a_3)a_1+\varphi^{-1}(k_{13}(a_1,a_3))$.
Similarly, using $(3.27)$ one can prove that $k_{34}$ is of the
desired form as well.
\end{proof}

\begin{lemma}\label{xxsec3.12}
There exist linear mapping $\gamma: A_4\rightarrow \mathcal{Z}(A_1)$
and bilinear mapping $\delta: A_1\times A_4\rightarrow
\mathcal{Z}(A_1)$ such that
$f_{14}(a_1,a_4)=\gamma(a_4)a_1+\delta(a_1,a_4)$.
\end{lemma}

\begin{proof}
Since $a_1\mapsto f_{14}(a_1,a_4)$ is a commuting mapping of $A_1$
for all $a_4\in A_4$, there exist mappings $\gamma: A_4\rightarrow
\mathcal{Z}(A_1)$ and $\delta: A_1\times A_4\rightarrow
\mathcal{Z}(A_1)$ such that
$$
f_{14}(a_1,a_4)=\gamma(a_4)a_1+\delta(a_1,a_4),
$$
where $\delta$ is $\mathcal{R}$-linear in the first argument. Let us
show that $\gamma$ is $\mathcal{R}$-linear and that $\delta$ is
$\mathcal{R}$-bilinear. It is easy to observe that
$$
f_{14}(a_1,a_4+b_4)=\gamma(a_4+b_4)a_1+\delta(a_1,a_4+b_4)
$$
and
$$
f_{14}(a_1,a_4)+f_{14}(a_1,b_4)=\gamma(a_4)a_1+\delta(a_1,a_4)+\gamma(b_4)a_1+\delta(a_1,b_4).
$$
for all for all $a_1\in A_1$ and $a_4,b_4\in A_4$. Therefore
$$
\big(\gamma(a_4+b_4)-\gamma(a_4)-\gamma(b_4)\big)a_1+\delta(a_1,a_4+b_4)-\delta(a_1,a_4)-\delta(a_1,b_4)=0
$$
for all $a_1\in A_1$ and $a_4,b_4\in A_4$. Note that both $\gamma$
and $\delta$ map into $\mathcal{Z}(A_1)$ and hence
$(\gamma(a_4+b_4)-\gamma(a_4)-\gamma(b_4))[a_1,b_1]=0$ for all $a_1,
b_1\in A_1$ and $a_4,b_4\in A_4$. Applying Lemma \ref{xxsec3.1}
yields that $\gamma$ is $\mathcal{R}$-linear. Consequently, $\delta$
is $\mathcal{R}$-linear in the second argument.
\end{proof}

\begin{lemma}\label{xxsec3.13}
$k_{14}(a_1,a_4)=\gamma'(a_1)a_4+\varphi(\delta(a_1,a_4))$, where
$\gamma'(a_1)=k_{14}(a_1,1)-\varphi(\delta(a_1,1))$.
\end{lemma}

\begin{proof}
By $(3.22)$ we know that
$g_{12}(1,a_2)=f_{11}(1,1)a_2-a_2k_{11}(1,1)$ for all $a_2\in A_2$.
On the other hand, the equations $(3.22)-(3.27)$ together with
$(3.15)$ imply that
$$
f_{14}(a_1,a_4)a_2+g_{12}(a_1,a_2)a_4=a_1g_{24}(a_2,a_4)+a_2k_{14}(a_1,a_4)
\eqno(3.29)
$$
for all $a_1\in A_1, a_2\in A_2, a_4\in A_4$. Let us set $a_1=1$ in
$(3.29)$. Then
$$
g_{24}(a_2,a_4)=a_2\big(\zeta
a_4+\varphi(f_{14}(1,a_4))-k_{14}(1,a_4)\big) \eqno(3.30)
$$
for all $a_2\in A_2, a_4\in A_4$, where
$\zeta=\varphi(f_{11}(1,1))-k_{11}(1,1)$. Similarly, using $(3.25)$
and $(3.29)$ we have
$$ g_{12}(a_1,a_2)=\big(\theta
a_1+\varphi^{-1}(k_{14}(a_1,1))-f_{14}(a_1,1)\big)a_2 \eqno(3.31)
$$
for all $a_1\in A_1, a_2\in A_2$, where
$\theta=\varphi^{-1}(k_{44}(1,1))-f_{44}(1,1)$. Now the equations
$(3.29)-(3.31)$ and Lemma \ref{xxsec3.12} jointly show that
$$
(\gamma(a_4)a_1+\delta(a_1,a_4))a_2+\big(\theta a_1+\varphi^{-1}(k_{14}(a_1,1))-f_{14}(a_1,1)\big)a_2a_4
$$
$$
=a_2k_{14}(a_1,a_4)+a_1a_2\big(\zeta
a_4+\varphi(f_{14}(1,a_4))-k_{14}(1,a_4)\big)
$$
for all $a_1\in A_1, a_2\in A_2, a_4\in A_4$. That is,
$$
a_1a_2\big((\zeta+\varphi(\gamma(1)-\theta)a_4+\varphi(\delta(1,a_4))-k_{14}(1,a_4)\big)
$$
$$
=a_2\big(\gamma'(a_1)a_4+\varphi(\delta(a_1,a_4))-k_{14}(a_1,a_4)\big)
\eqno(3.32)
$$
for all $a_1\in A_1, a_2\in A_2, a_4\in A_4$. Replacing $a_2$ by
$b_1a_2$ in $(3.32)$ and subtracting the left multiplication of
$(3.32)$ by $b_1$ gives
$$
[a_1,b_1]a_2\big((\zeta+\varphi(\gamma(1)-\theta)a_4+\varphi(\delta(1,a_4))-k_{14}(1,a_4)\big)=0
$$
for all $a_1, b_1\in A_1, a_2\in A_2, a_4\in A_4$. Note that $M=A_2$
is loyal and $A=A_1$ is noncommutative. It follows that
$$
k_{14}(1,a_4)=(\zeta+\varphi(\gamma(1)-\theta)a_4+\varphi(\delta(1,a_4))
$$
for all $a_4\in A_4$. Consequently, the relation $(3.32)$ implies
that
$$
A_2\big(\gamma'(a_1)a_4+\varphi(\delta(a_1,a_4))-k_{14}(a_1,a_4)\big)=0
$$
for all $a_1, a_4\in A_4$. Since $A_2=M$ is faithful as a right
$B$-module, $k_{14}$ is of the desired form.
\end{proof}

{\noindent}{\bf Proof of Theorem \ref{xxsec3.4}:} Let us write
$\varepsilon=\theta-\gamma(1)$ and $\varepsilon'=\zeta-\gamma'(1)$.
By the equations $(3.30)$ and $(3.31)$ and the form of $f_{14},
k_{14}$, we have the following relations:
$$
g_{12}(a_1,a_2)=\varepsilon a_1a_2+\varphi^{-1}(\gamma'(a_1))a_2,
\quad g_{24}(a_2,a_4)=a_2(\varepsilon' a_4+\varphi(\gamma(a_4)))
\eqno(3.33)
$$
for all $a_1\in A_1, a_2\in A_2, a_4\in A_4$. By $(3.1)$ and those
similar computational procedures we get
$$
h_{13}(a_1,a_3)=a_3\varepsilon a_1+\gamma'(a_1)a_3, \quad
h_{34}(a_3,a_4)=\varepsilon' a_4a_3+\varphi(\gamma(a_4))a_3
\eqno(3.34)
$$
for all $a_1\in A_1, a_3\in A_3, a_4\in A_4$. Taking $a_1=1$ and
$a_4=1$ into $(3.29)$ and combining Lemma \ref{xxsec3.12}, Lemma
\ref{xxsec3.13} with $(3.33)$, we conclude that $\varepsilon
a_2=a_2\varepsilon'$ for all $a_2\in A_2$. Note that $\varepsilon\in
\mathcal{Z}(A_1)=\pi_A(\mathcal{Z(G)})$ and $\varepsilon'\in
\mathcal{Z}(A_4)=\pi_B(\mathcal{Z(G)})$. In view of \cite[Lemma
3.2]{XiaoWei1} we obtain $\left[
\begin{array}
[c]{cc}%
\varepsilon & 0\\
0 & \varepsilon'
\end{array}
\right]\in \mathcal{Z(G)}$.

It follows from $(3.22)$ and $(3.33)$ that
$$
\big(f_{11}(a_1,a_1)-\varepsilon a_1^2-\varphi^{-1}(\gamma'(a_1))a_1-\varphi^{-1}(k_{11}(a_1,a_1))\big)a_2=0
$$
for all $a_1\in A_1, a_2\in A_2$. Since $A_2=M$ is faithful as a
left $A$-module,
$$
f_{11}(a_1,a_1)=\varepsilon
a_1^2+\varphi^{-1}(\gamma'(a_1))a_1+\varphi^{-1}(k_{11}(a_1,a_1))\eqno(3.35)
$$
for all $a_1\in A_1$. Similarly,
$$
k_{44}(a_4,a_4)=\varepsilon'
a_4^2+\varphi(\gamma(a_4))a_4+\varphi(f_{44}(a_4,a_4))\eqno(3.36)
$$
for all $a_4\in A_4$.

Finally, let us set $z=\left[
\begin{array}
[c]{cc}%
\varepsilon & 0\\
0 & \varepsilon'
\end{array}
\right]$ and define the mapping $\mu: \mathcal{G}\rightarrow
\mathcal{Z(G)}$ by
$$
\begin{aligned}
&\left[
\begin{array}
[c]{cc}%
a_1 & a_2\\
a_3 & a_4\\
\end{array}
\right]\mapsto \\
&\left[
\begin{array}
[c]{cc}%
\varphi^{-1}(\gamma'(a_1))+\gamma(a_4)+\alpha(a_2)+\tau(a_3) & 0\\
0 & \gamma'(a_1)+\varphi(\gamma(a_4)+\alpha(a_2)+\tau(a_3))\\
\end{array}
\right].
\end{aligned}
$$
In view of all conclusions derived above, we see that
$$
\begin{aligned}
\nu(x):&={\mathfrak T}_{\mathfrak q}(x)-z x^2-\mu(x)x\\
&\equiv \left[
\begin{array}
[c]{cc}%
f_{23}(a_2,a_3)-\varepsilon a_2a_3 & 0\\
0 & k_{23}(a_2,a_3)-\varepsilon' a_3a_2\\
\end{array}
\right]\quad ({\rm mod} \hspace{2pt}\mathcal{Z(G)})
\end{aligned}
$$
where $x=\left[
\begin{array}
[c]{cc}%
a_1 & a_2\\
a_3 & a_4\\
\end{array}
\right]$. Therefore we can write
$$
\mathfrak{T_q}(x)=zx^2+\mu(x)x+\left[
\begin{array}
[c]{cc}%
f_{23}(a_2,a_3)-\varepsilon a_2a_3 & 0\\
0 & k_{23}(a_2,a_3)-\varepsilon' a_3a_2\\
\end{array}
\right]+c
$$
for some $c\in \mathcal{Z(G)}$. Since $\mathfrak{q}$ is a commuting
mapping, we have
$$
\left[\left[
\begin{array}
[c]{cc}%
f_{23}(a_2,a_3)-\varepsilon a_2a_3 & 0\\
0 & k_{23}(a_2,a_3)-\varepsilon' a_3a_2\\
\end{array}
\right] ,\left[
\begin{array}
[c]{cc}%
a_1 & a_2\\
a_3 & a_4\\
\end{array}
\right] \right]=0.
$$
This implies that $f_{23}(a_2,a_3)-\varepsilon a_2a_3\in {\mathcal
Z}(A_1)=\pi_A(\mathcal{Z(G)})$ and $k_{23}(a_2,a_3)-\varepsilon'
a_3a_2\in {\mathcal Z}(A_4)=\pi_B(\mathcal{Z(G)})$. Moreover, it
shows that
$$
(f_{23}(a_2,a_3)-\varepsilon
a_2a_3)a_2=a_2(k_{23}(a_2,a_3)-\varepsilon' a_3a_2)
$$
and
$$
a_3(f_{23}(a_2,a_3)-\varepsilon a_2a_3)=(k_{23}(a_2,a_3)-\varepsilon' a_3a_2)a_3.
$$
for all $a_2\in A_2, a_3\in A_3$. For convenience, let us write
$f(a_2,a_3)=f_{23}(a_2,a_3)-\varepsilon a_2a_3$ and
$k(a_2,a_3)=k_{23}(a_2,a_3)-\varepsilon' a_3a_2$. Thus
$$
\big(f(a_2,a_3)-\varphi^{-1}(k(a_2,a_3)) \big)a_2=0
$$
for all $a_2\in A_2, a_3\in A_3$. A linearization of the last
relation gives
$$
\big(f(a_2,a_3)-\varphi^{-1}(k(a_2,a_3)) \big)b_2+
\big(f(b_2,a_3)-\varphi^{-1}(k(b_2,a_3)) \big)a_2=0
$$
for all $a_2, b_2\in A_2, a_3\in A_3$. Note that the hypothesis
$A_2=M$ is loyal as an $(A,B)$-bimodule. It follows from \cite[Lemma
2.3]{BenkovicEremita} that $f(a_2,a_3)-\varphi^{-1}(k(a_2,a_3))=0$
for all $a_2\in A_2, a_3\in A_3$. Hence $\nu$ maps $\mathcal{G}$
into $\mathcal{Z(G)}$ and this completes the proof of the
theorem.\qed \vspace{6pt}

As a direct consequence of Theorem \ref{xxsec3.4} we get

\begin{corollary}\cite[Theorem 3.1]{BenkovicEremita}\label{xxsec3.14}
Let $\mathcal{T}=\left[
\begin{array}
[c]{cc}%
A & M\\
O & B\\
\end{array}
\right]$ be a $2$-torsionfree triangular algebra over a commutative
ring $\mathcal{R}$ and ${\mathfrak q}\colon \mathcal{T}\times
\mathcal{T}\longrightarrow \mathcal{T}$ be an $\mathcal{R}$-bilinear
mapping. If
\begin{enumerate}
\item[{\rm(1)}] every commuting linear mapping on $A$ or $B$ is proper;
\item[{\rm(2)}] $\pi_A(\mathcal{Z(T)})=\mathcal{Z}(A)\neq A$ and
$\pi_B(\mathcal{Z(T)})=\mathcal{Z}(B)\neq B$;
\item[{\rm(3)}] $M$ is loyal,
\end{enumerate}
then every commuting trace ${\mathfrak T}_{\mathfrak q}:
\mathcal{T}\longrightarrow \mathcal{T}$ of ${\mathfrak q}$ is
proper.
\end{corollary}

In particular, we also have

\begin{corollary}\label{xxsec3.15}\cite[Corollary 3.4]{BenkovicEremita}
Let $n\geq 2$ and $\mathcal{R}$ be a $2$-torsionfree commutative
domain. Suppose that ${\mathfrak q}\colon  {\mathcal
T}_n(\mathcal{R})\times {\mathcal T}_n(\mathcal{R})\longrightarrow
{\mathcal T}_n(\mathcal{R})$ is an $\mathcal{R}$-bilinear mapping.
Then every commuting trace ${\mathfrak T}_{\mathfrak q}\colon
{\mathcal T}_n(\mathcal{R})\longrightarrow {\mathcal
T}_n(\mathcal{R})$ of $\mathfrak{q}$ is proper.
\end{corollary}

\begin{corollary}\label{xxsec3.16}\cite[Corollary 3.5]{BenkovicEremita}
Let $\mathcal{N}$ be a nest of a Hilbert space $\mathbf{H}$. Suppose
that ${\mathfrak q}\colon  {\mathcal T}(\mathcal{N})\times {\mathcal
T}(\mathcal{N})\longrightarrow {\mathcal T}(\mathcal{N})$ is an
$\mathcal{R}$-bilinear mapping. Then every commuting trace
${\mathfrak T}_{\mathfrak q}\colon {\mathcal
T}(\mathcal{N})\longrightarrow {\mathcal T}(\mathcal{N})$ of
$\mathfrak{q}$ is proper.
\end{corollary}

In order to handle the commuting traces of bilinear mappings on full
matrix algebras we need a technical lemma in below. Recall that an
algebra $\mathcal{A}$ over a commutative ring $\mathcal{R}$ is said
to be {\em central over} $\mathcal{R}$ if
$\mathcal{Z(A)}=\mathcal{R}1$.

\begin{proposition}\label{xxsec3.17}
Let $\mathcal{G}=\left[
\begin{array}
[c]{cc}%
\mathcal{R} & M\\
N & B\\
\end{array}
\right]$ be a $2$-torsionfree generalized matrix algebra over a
commutative ring $\mathcal{R}$, where $B$ is a noncommutative
algebra over $\mathcal{R}$ and both $\mathcal{G}$ and $B$ are
central over $\mathcal{R}$. Suppose that ${\mathfrak q}\colon
\mathcal{G}\times \mathcal{G}\longrightarrow \mathcal{G}$ is an
$\mathcal{R}$-bilinear mapping. If
\begin{enumerate}
\item[{\rm(1)}] every commuting linear mapping of $B$ is proper,
\item[{\rm(2)}] for any $r\in\mathcal{R}$ and $m\in M$, $rm=0$
implies that $r=0$ or $m=0$,
\item[{\rm(3)}] there exist $m_0\in M$ and $b_0\in B$ such that
$m_0b_0$ and $m_0$ are $\mathcal{R}$-linearly independent,
\end{enumerate}
then every commuting trace ${\mathfrak T}_{\mathfrak q}:
\mathcal{G}\longrightarrow \mathcal{G}$ of $\mathfrak{q}$ is proper.
\end{proposition}

\begin{proof}
For the proof of this lemma, we shall follow the proof of Theorem
\ref{xxsec3.4} step by step and hence use the same notations.
However, we have to make explicit changes in some necessary places.
All changes take place from the Lemma \ref{xxsec3.9} to the end.

{\bf Step 1.} $\left[
\begin{array}
[c]{cc}%
f_{22}(a_2,a_2) & 0\\
0 & k_{22}(a_2,a_2)
\end{array}
\right] \in \mathcal{R}1$ and $\left[
\begin{array}
[c]{cc}%
f_{33}(a_3,a_3) & 0\\
0 & k_{33}(a_3,a_3)
\end{array}
\right] \in \mathcal{R}1$. By $(3.17)$ we know that
$$
\big(f_{22}(a_2,a_2)-\varphi^{-1}(k_{22}(a_2,a_2)) \big)a_2=0
$$
for all $a_2\in A_2=M$. Note that the fact $A_1=\mathcal{R}$ in our
context. Then the assumption $(2)$ deduces that
$f_{22}(a_2,a_2)=\varphi^{-1}(k_{22}(a_2,a_2))$. Using the same
proof of Lemma \ref{xxsec3.9} one easily obtain $\left[
\begin{array}
[c]{cc}%
f_{22}(a_2,a_2) & 0\\
0 & k_{22}(a_2,a_2)
\end{array}
\right] \in \mathcal{R}1$. On the other hand, $\left[
\begin{array}
[c]{cc}%
f_{33}(a_3,a_3) & 0\\
0 & k_{33}(a_3,a_3)
\end{array}
\right] \in \mathcal{R}1$ follows from the second paragraph of the
proof of Lemma \ref{xxsec3.9}.

{\bf Step 2.}
$f_{12}(a_1,a_2)=\alpha(a_2)a_1+\varphi^{-1}(k_{12}(a_1,a_2))$ and
$k_{24}(a_2,a_4)=\varphi(\alpha(a_2))a_4+\varphi (f_{24}(a_2,a_4))$,
where $\alpha(a_2)=f_{12}(1,a_2)-\varphi^{-1}(k_{12}(1,a_2))$. If
only we show that $M$ is loyal as an $(A_1,A_4)$-bimodule, then the
corresponding form of $f_{12}$ can be obtained by copying the proof
of Lemma \ref{xxsec3.10}. Let $rMb=0$ for all $r\in \mathcal{R}$ and
$b\in B$. Suppose that $b\neq 0$. Since $M$ is faithful as a right
$B$-module, there exists a $m\in M$ such that $mb\neq 0$. However
$0=rmb=r(mb)$, the assumption $(2)$ implies that $r=0$. And hence
$M$ is a loyal $(A_1,A_4)$-bimodule.

It is necessary for us to characterize the form of $k_{24}$. By
equations $(3.26)$ and $(3.28)$ we see that
$$
a_2\big(k_{24}(a_2,a_4)-\varphi(\alpha(a_2))a_4-\varphi(f_{24}(a_2,a_4))
\big)=0 \eqno(3.37)
$$
for all $a_i\in A_i$ with $i=1,2,4$. Since $a_4\mapsto
k_{24}(a_2,a_4)$ is a commuting linear mapping on $A_4$, there exist
mappings $\psi: A_2\longrightarrow \mathcal{R}1$ and $\omega:
A_2\times A_4\longrightarrow \mathcal{R}1$ such that
$$k_{24}(a_2,a_4)=\psi(a_2)a_4+\omega(a_2,a_4),$$
where $\omega$ is $\mathcal{R}$-linear in the second argument. Let
us prove that $\psi$ is an $\mathcal{R}$-linear mappings and that
$\omega$ is an $\mathcal{R}$-bilinear mapping. It is straightforward
to check that
$$
k_{24}(a_2+b_2,a_4)=\psi(a_2+b_2)a_4+\omega(a_2+b_2,a_4)
$$
and
$$
k_{24}(a_2,a_4)+k_{24}(b_2,a_4)=\psi(a_2)a_4+\omega(a_2,a_4)+\psi(b_2)a_4+\omega(b_2,a_4).
$$
for all $a_2,b_2\in A_2$ and $a_4\in A_4$. Therefore
$$
\big(\psi(a_2+b_2)-\psi(a_2)-\psi(b_2)\big)a_4+\omega(a_2+b_2,a_4)-\omega(a_2,a_4)-\omega(b_2,a_4)=0
\eqno(3.38)
$$
for all $a_2,b_2\in A_2$ and $a_4\in A_4$. Note that both $\psi$ and
$\omega$ map into ${\mathcal Z}(A_4)$. Commuting $(3.38)$ with
$b_4\in A_4$ we get
$$
(\psi(a_2+b_2)-\psi(a_2)-\psi(b_2))[a_4,b_4]=0
$$
for all $a_2, b_2\in A_2$ and $a_4,b_4\in A_4$. Let us choose
$a_4,b_4\in A_4$ such that $[a_4,b_4]\neq 0$. Since $M$ is faithful
as a right $A_4$-module, there exists $m\in M$ such that
$m[a_4,b_4]\neq 0$. Thus
$$
\varphi^{-1}\big(\psi(a_2+b_2)-\psi(a_2)-\psi(b_2) \big)m[a_4,b_4]=0
$$
for all $a_2,b_2\in A_2$. The assumption $(2)$ implies that $\psi$
is an $\mathcal{R}$-linear mapping. Consequently, $\omega$ is
$\mathcal{R}$-linear in the first argument. Rewrite $(3.37)$ as
$$
a_2\big((\psi(a_2)-\varphi(\alpha(a_2)))a_4+\omega(a_2,a_4)-\varphi(f_{24}(a_2,a_4))
\big)=0 \eqno(3.39)
$$
for all $a_2\in A_2$ and $a_4\in A_4$. Setting $a_2=m_0$ and $a_4=b_0$ we obtain
$$
\big(\varphi^{-1}(\psi(m_0))-\alpha(m_0) \big)m_0b_0+\big(\varphi^{-1}(\omega(m_0,b_0))-f_{24}(m_0,b_0) \big)m_0=0.
$$
So $\alpha(m_0)=\varphi^{-1}(\psi(m_0))$ and
$f_{24}(m_0,b_0)=\varphi^{-1}(\omega(m_0,b_0))$ by the condition
$(3)$. Substituting $a_2+m_0$ for $a_2$ and $b_0$ for $a_4$ in
$(3.39)$ yields
$$
\big(\varphi^{-1}(\psi(a_2))-\alpha(a_2) \big)m_0b_0+\big(\varphi^{-1}(\omega(a_2,b_0))-f_{24}(a_2,b_0) \big)m_0=0.
$$
Therefore $\alpha(a_2)=\varphi^{-1}(\psi(a_2))$ for all $a_2\in
A_2$. Then it follows from $(3.39)$ that
$\omega(a_2,a_4)=\varphi(f_{24}(a_2,a_4))$ for all $a_2\in A_2,
a_4\in A_4$. Hence $k_{24}$ has also the desired form.

Since $M$ is loyal, we only need to change the places in the proof
of Theorem \ref{xxsec3.4}, where the noncommutativity of $A$ is
involved. However, the proof of Lemma \ref{xxsec3.11} does not
involve the noncommutativity of $A$ and hence it still works in our
context.

{\bf Step 3.} $f_{14}$ (resp. $k_{14}$) is of the form as in Lemma
\ref{xxsec3.12} (resp. Lemma \ref{xxsec3.13}). Note that $a_4\mapsto
k_{14}(a_1,a_4)$ is a commuting $\mathcal{R}$-linear mapping on
$A_4$. Then there exist mappings $\gamma': A_1\rightarrow
\mathcal{R}1_B$ and $\delta': A_1\times A_4\rightarrow
\mathcal{R}1_B$ such that
$$
k_{14}(a_1,a_4)=\gamma'(a_1)a_4+\delta'(a_1,a_4), \eqno(3.40)
$$
where $\delta'$ is $\mathcal{R}$-linear in the second argument. Here
we denote $1_B$ the identity of $B$ to avoid confusion in the
following discussion. We assert that $\gamma'$ is an
$\mathcal{R}$-linear mapping and $\delta'$ is an
$\mathcal{R}$-bilinear mapping. In fact,
$k_{14}(1,a_4)=\gamma'(1)a_4+\delta'(1,a_4)$ and hence
$k_{14}(a_1,a_4)=a_1\gamma'(1)a_4+a_1\delta'(1,a_4)$. Therefore
$$
(\gamma'(a_1)-a_1\gamma'(1))a_4+\delta'(a_1,a_4)-a_1\delta'(1,a_4)=0
\eqno(3.41)
$$
for all $a_1\in \mathcal{R}$, $a_4\in A_4$. Commuting $(3.41)$ with
$b_4\in A_4$ we obtain
$$
(\gamma'(a_1)-a_1\gamma'(1))[a_4,b_4]=0
$$
for all $a_1\in \mathcal{R}$, $a_4,b_4\in A_4$. Moreover,
$$
\varphi^{-1}(\gamma'(a_1)-a_1\gamma'(1))m[a_4,b_4]=0
$$
for all $a_1\in \mathcal{R}$, $a_4,b_4\in A_4$ and $m\in M$. Since
$M$ is loyal and $B$ is noncommutative, we have
$\gamma'(a_1)=a_1\gamma'(1)$. This implies that $\gamma^\prime$ is
$\mathcal{R}$-linear and hence $\delta'$ is $\mathcal{R}$-bilinear.

It would be helpful to point out here that each of the mappings
$f_{ij}$ takes its values in $\mathcal{R}$. Now the identities
$(3.29), (3.30)$ and $(3.31)$ jointly yield that
$$
f_{14}(a_1,a_4)a_2+\big(\theta a_1+\varphi^{-1}(k_{14}(a_1,1_B))-f_{14}(a_1,1_B) \big)a_2a_4
$$
$$
=a_2k_{14}(a_1,a_4)+a_1a_2\big(\eta a_4+\varphi(f_{14}(1,a_4))-k_{14}(1,a_4) \big)
$$
and hence (taking into account the relation $(3.40)$)
$$
a_2\left\{\varphi\big(a_1\varphi^{-1}(\eta)+\varphi^{-1}(\gamma'(a_1-1)-k_{14}(a_1,1_B))-\theta a_1+f_{14}(a_1,1_B) \big)a_4 \right.
$$
$$
+\left.\varphi\big((f_{14}(1,a_4)-\varphi^{-1}(\delta'(1,a_4)))a_1+\varphi^{-1}(\delta'(a_1,a_4))-f_{14}(a_1,a_4)
\big)\right\}=0 \eqno(3.42)
$$
for all $a_i\in A_i$ with $i=1,2,4$. Let us choose $a_4,b_4\in A_4$
such that $[a_4,b_4]\neq 0$. Then the fact $A_2$ is faithful as a
right $A_4$-module and the relation $(3.42)$ deduce that
$$
\varphi\big(a_1\varphi^{-1}(\eta)+\varphi^{-1}(\gamma'(a_1-1)-k_{14}(a_1,1_B))-\theta
a_1+f_{14}(a_1,1_B) \big)[a_4,b_4]=0.
$$
for all $a_1\in A_1$. Thus
$$
\big(a_1\varphi^{-1}(\eta)+\varphi^{-1}(\gamma'(a_1-1)-k_{14}(a_1,1_B))-\theta a_1+f_{14}(a_1,1_B) \big)M[a_4,b_4]=0
$$
for all $a_1\in A_1$. Since $M$ is faithful as a right $B$-module,
there exists a $m\in M$ such that $m[a_4,b_4]\neq 0$. Therefore the
condition $(2)$ implies that
$$
a_1\varphi^{-1}(\eta)+\varphi^{-1}(\gamma'(a_1-1)-k_{14}(a_1,1_B))-\theta a_1+f_{14}(a_1,1_B)=0
$$
for all $a_1\in A_1$. Then the relation $(3.42)$ shows
$$
(f_{14}(1,a_4)-\varphi^{-1}(\delta'(1,a_4)))a_1+\varphi^{-1}(\delta'(a_1,a_4))=f_{14}(a_1,a_4)
$$
for all $a_1\in A_1, a_4\in A_4$. Let us
$\gamma(a_4):=f_{14}(1,a_4)-\varphi^{-1}(\delta'(1,a_4))$ and
$\delta(a_1,a_4):=\varphi^{-1}(\delta'(a_1,a_4))$. Then
$f_{14}(a_1,a_4)=\gamma(a_4)a_1+\delta(a_1,a_4)$ and
$k_{14}(a_1,a_4)=\gamma'(a_1)a_4+\varphi(\delta(a_1,a_4))$

Finally, following the rest part of the proof of Theorem
\ref{xxsec3.4} we can obtain the required result.
\end{proof}

\begin{corollary}\label{xxsec3.18}
Let $\mathcal{R}$ be a $2$-torsionfree commutative domain and
${\mathcal M}_n(\mathcal{R})$ be the full matrix algebra over
$\mathcal{R}$. Suppose that ${\mathfrak q}\colon  {\mathcal
M}_n(\mathcal{R})\times {\mathcal M}_n(\mathcal{R})\longrightarrow
{\mathcal M}_n(\mathcal{R})$ is an $\mathcal{R}$-bilinear mapping.
Then every commuting trace ${\mathfrak T}_{\mathfrak q}\colon
{\mathcal M}_n(\mathcal{R})\longrightarrow {\mathcal
M}_n(\mathcal{R})$ of $\mathfrak{q}$ is proper.
\end{corollary}

\begin{proof}
If $n>3$, then ${\mathcal M}_n(\mathcal{R})=\left[
\begin{array}
[c]{cc}%
{\mathcal
M}_2(\mathcal{R}) & {\mathcal M}_{2\times (n-2)}(\mathcal{R})\\
{\mathcal M}_{(n-2)\times 2}(\mathcal{R}) & {\mathcal
M}_{n-2}(\mathcal{R})\\
\end{array}
\right]$. By \cite[Corollary 4.1]{XiaoWei1} we know that each
commuting linear mapping on ${\mathcal M}_2(\mathcal{R})$ and
${\mathcal M}_{n-2}(\mathcal{R})$ is proper. The assumptions $(2)$
and $(3)$ in Theorem \ref{xxsec3.4} clearly holds for ${\mathcal
M}_n(\mathcal{R})$ $(n>3)$. Applying Theorem \ref{xxsec3.4} yields
the desired conclusion.

If $n=3$, then ${\mathcal M}_3(\mathcal{R})=\left[
\begin{array}
[c]{cc}%
{\mathcal
 R} & {\mathcal M}_{1\times 2}(\mathcal{R})\\
{\mathcal M}_{2\times 1}(\mathcal{R}) & {\mathcal
M}_2(\mathcal{R})\\
\end{array}
\right]$. Therefore there exist elements
$$
m_0=\left[1,0\right]\in {\mathcal M}_{1\times 2}(\mathcal{R}) \quad
\text{and} \quad b_0=\left[
\begin{array}
[c]{cc}%
0 & 1\\
0 & 0\\
\end{array}
\right]\in {\mathcal M}_{2}(\mathcal{R})
$$
such that $m_0b_0$ and $m_0$ are linearly independent over
$\mathcal{R}$. By \cite[Corollary 4.1]{XiaoWei1} and Proposition
\ref{xxsec3.17} we conclude that ${\mathfrak T}_{\mathfrak q}$ has
the proper form.

If $n=2$, the result follows from \cite[Theorem 3.1]{BresarSemrl}.

Finally, if $n=1$, the conclusion is obvious.
\end{proof}

\begin{corollary}\label{xxsec3.19}
Let $\mathcal{R}$ be a $2$-torsionfree commutative domain, $V$ be an
$\mathcal{R}$-linear space and $B(\mathcal{R}, V, \gamma)$ be the
inflated algebra of $\mathcal{R}$ along $V$. Suppose that
${\mathfrak q}\colon  B(\mathcal{R}, V, \gamma)$ $\times
B(\mathcal{R}, V, \gamma)\longrightarrow B(\mathcal{R}, V, \gamma)$
is an $\mathcal{R}$-bilinear mapping. Then every commuting trace
${\mathfrak T}_{\mathfrak q}\colon B(\mathcal{R}, V,
\gamma)\longrightarrow B(\mathcal{R}, V, \gamma)$ of $\mathfrak{q}$
is proper.
\end{corollary}

Let us see the commuting traces of bilinear mappings of several
unital algebras with nontrivial idempotents.

\begin{corollary}\label{xxsec3.20}
Let $\mathcal{A}$ be a $2$-torsionfree unital prime algebra over a
commutative ring $\mathcal{R}$. Suppose that $\mathcal{A}$ contains
a nontrivial idempotent $e$ and that $f=1-e$. If
$e\mathcal{Z(A)}e=\mathcal{Z}(e\mathcal{A}e)\neq e\mathcal{A}e$ and
$f\mathcal{Z(A)}f=\mathcal{Z}(f\mathcal{A}f)\neq f\mathcal{A}f$,
then every commuting trace of an arbitrary bilinear mappings on
$\mathcal{A}$ is proper.
\end{corollary}

\begin{proof}
Let us write $\mathcal{A}$ as a natural generalized matrix algebra
$\left[
\begin{array}
[c]{cc}%
e\mathcal{A}e & e\mathcal{A}f\\
f\mathcal{A}e & f\mathcal{A}f\\
\end{array}
\right]$. It is clear that $e\mathcal{A}e$ and $f\mathcal{A}f$ are
prime algebras. By \cite[Theorem 3.2]{Bresar2} it follows that each
commuting additive mapping on $e\mathcal{A}e$ and $f\mathcal{A}f$ is
proper. On the other hand, if $(eae)e\mathcal{A}f(fbf)=0$ holds for
all $a,b\in \mathcal{A}$, then the primeness of $\mathcal{A}$
implies that $eae=0$ or $fbf=0$. This shows $e\mathcal{A}f$ is a
loyal $(e\mathcal{A}e, f\mathcal{A}f)$-bimodule. Applying Theorem
\ref{xxsec3.4} yields that each commuting trace of an arbitrary
bilinear mappings on $\mathcal{A}$ is proper.
\end{proof}

\begin{corollary}\label{xxsec3.21}
Let $X$ be a Banach space over the real or complex field
$\mathbb{F}$, $\mathcal{B}(X)$ be the algebra of all bounded linear
operators on $X$. Then every commuting trace of an arbitrary
bilinear mapping on $\mathcal{B}(X)$ is proper.
\end{corollary}

\begin{proof}
Note that $\mathcal{B}(X)$ is a centrally closed prime algebra. If
$X$ is infinite dimensional, the result follows from Corollary
\ref{xxsec3.20}. If $X$ is of dimension $n$, then $\mathcal{B}(X)=
{\mathcal M}_n(\mathbb{F})$. In this case the result follows from
Corollary \ref{xxsec3.18}.
\end{proof}

\section{Lie Isomorphisms on Generalized Matrix Algebras}\label{xxsec4}

In this section we shall use the main result in Section \ref{xxsec3}
(Theorem \ref{xxsec3.4}) to describe the form of an arbitrary Lie
isomorphism of a certain class of generalized matrix algebras
(Theorem \ref{xxsec4.3}). As applications of Theorem \ref{xxsec4.3},
we characterize Lie isomorphisms of certain generalized matrix
algebras. The involved algebras include upper triangular matrix
algebras, nest algebras, full matrix algebras, inflated algebras,
prime algebras with nontrivial idempotents.

Throughout this section, we denote the generalized matrix algebra of
order $2$ originated from the Morita context $(A, B, _AM_B, _BN_A,
\Phi_{MN}, \Psi_{NM})$ by
$$
\mathcal{G}=\left[
\begin{array}
[c]{cc}%
A & M\\
N & B
\end{array}
\right] ,
$$
where at least one of the two bimodules $M$ and $N$ is distinct from
zero. We always assume that $M$ is faithful as a left $A$-module and
also as a right $B$-module, but no any constraint conditions on $N$.

\begin{lemma}\label{xxsec4.1}
Let $\mathcal{G}=\left[
\begin{array}
[c]{cc}%
A & M\\
N & B\\
\end{array}
\right]$ be a $2$-torsionfree generalized matrix algebra over a
commutative ring $\mathcal{R}$. Then $\mathcal{G}$ satisfies the
polynomial identity $[[x^2,y],[x,y]]$ if and only if both $A$ and
$B$ are commutative.
\end{lemma}

\begin{proof}
If $A$ and $B$ are commutative, then we can prove that $\mathcal{G}$
satisfies the polynomial identity $[[x^2,y],[x,y]]$ by a direct but
rigorous procedure.

The necessity can be obtained by a similar proof of \cite[Lemma
2.7]{BenkovicEremita}.
\end{proof}

The following proposition is a much more common generalization of
\cite[Lemma 4.1]{BenkovicEremita}. We here give out the proof for
completeness and for reading convenience.

\begin{proposition}\label{xxsec4.2}
Let $\mathcal{G}=\left[
\begin{array}
[c]{cc}%
A & M\\
N & B\\
\end{array}
\right]$ and $\mathcal{G}^\prime=\left[
\begin{array}
[c]{cc}%
A^\prime & M^\prime\\
N^\prime & B^\prime\\
\end{array}
\right]$ be generalized matrix algebras over $\mathcal{R}$ with
$1/2\in\mathcal{R}$. Let $\mathfrak{l}\colon
\mathcal{G}\longrightarrow\mathcal{G'}$ be a Lie isomorphism. If
\begin{enumerate}
\item[{\rm(1)}] every commuting trace of an arbitrary bilinear mapping on $\mathcal{G'}$ is proper,
\item[{\rm(2)}] at least one of $A, B$ and at least one of $A^\prime, B^\prime$ are noncommutative,
\item[{\rm(3)}] $M^\prime$ is loyal,
\end{enumerate}
then $\mathfrak{l}=\mathfrak{m}+\mathfrak{n}$, where
$\mathfrak{m}\colon: \mathcal{G}\longrightarrow\mathcal{G^\prime}$
is a homomorphism or the negative of an anti-homomorphism,
$\mathfrak{m}$ is injective, and $\mathfrak{n}\colon
\mathcal{G}\longrightarrow {\mathcal Z}(\mathcal{G}^\prime)$ is a
linear mapping vanishing on each commutator. Moreover, if
$\mathcal{G^\prime}$ is central over $\mathcal{R}$, then
$\mathfrak{m}$ is surjective.
\end{proposition}

\begin{proof}
Clearly $[\mathfrak{l}(x),\mathfrak{l}(x^2)]=0$ for all $x\in
\mathcal{G}$. Replacing $x$ by $\mathfrak{l}^{-1}(y)$, we get
$[y,\mathfrak{l}(\mathfrak{l}^{-1}(y)^2)]=0$ for all $y\in
\mathcal{G}^\prime$. This means that the mapping
$\mathfrak{T}_{\mathfrak
q}(y):=\mathfrak{l}(\mathfrak{l}^{-1}(y)^2)$ is commuting. Since
$\mathfrak{T}_{\mathfrak q}$ is also a trace of the bilinear mapping
${\mathfrak q}\colon
\mathcal{G'}\times\mathcal{G'}\longrightarrow\mathcal{G'}$,
${\mathfrak q}(y,z):={\mathfrak l}({\mathfrak l}^{-1}(y){\mathfrak
l}^{-1}(z))$, by the hypothesis there exist $\lambda\in
\mathcal{Z(G^\prime)}$, a linear mapping $\mu_1:
\mathcal{G'}\longrightarrow \mathcal{Z(G^\prime)}$ and a trace
$\nu_1: \mathcal{G'}\longrightarrow \mathcal{Z(G^\prime)}$ of a
bilinear mapping such that
$$
\mathfrak{l}\big(\mathfrak{l}^{-1}(y)^2\big)=\lambda
y^2+\mu_1(y)y+\nu_1(y) \eqno(4.1)
$$
for all $y\in \mathcal{G}^\prime$. Let $\mu=\mu_1\mathfrak{l}$ and
$\nu=\nu_1\mathfrak{l}$. Then $\mu$ and $\nu$ are mappings of
$\mathcal{G}$ into $\mathcal{Z(G^\prime)}$ and $\mu$ is linear.
Hence $(4.1)$ can be rewritten as
$$
\mathfrak{l}(x^2)=\lambda
\mathfrak{l}(x)^2+\mu(x)\mathfrak{l}(x)+\nu(x) \eqno(4.2)
$$
for all $x\in \mathcal{G}$.

We assert that $\lambda\neq 0$. Otherwise we have
$\mathfrak{l}(x^2)-\mu(x)\mathfrak{l}(x)\in \mathcal{Z(G^\prime)}$
by $(4.2)$ and hence
$$
\mathfrak{l}([[x^2,y],[x,y]])=[[\mu(x)\mathfrak{l}(x),\mathfrak{l}(y)],[\mathfrak{l}(x),\mathfrak{l}(y)]]=0
$$
for all $x,y\in \mathcal{G}$. Consequently, $[[x^2,y],[x,y]]=0$ for
all $x,y\in \mathcal{G}$, which is contradictory to Lemma
\ref{xxsec4.1} by our assumptions.

Now we define a linear mapping $\mathfrak{m}\colon
\mathcal{G}\longrightarrow \mathcal{G'}$ by
$$
\mathfrak{m}(x):=\lambda\mathfrak{l}(x)+\frac{1}{2}\mu(x).\eqno(4.3)
$$
In view of $(4.2)$ we have
$$
\mathfrak{m}(x^2)=\lambda\mathfrak{l}(x^2)+\frac{1}{2}\mu(x^2)=\lambda^2\mathfrak{l}(x)^2+\lambda\mu(x)\mathfrak{l}(x)+
\lambda\nu(x)+\frac{1}{2}\mu(x^2).
$$
On the other hand,
$$
\mathfrak{m}(x)^2=\left(\lambda\mathfrak{l}(x)+\frac{1}{2}\mu(x)\right)^2=\lambda^2\mathfrak{l}(x)^2
+\lambda\mu(x)\mathfrak{l}(x)+\frac{1}{4}\mu(x)^2.
$$
Comparing the above two identities we get
$$
\mathfrak{m}(x^2)-\mathfrak{m}(x)^2\in \mathcal{Z(G^\prime)}
\eqno(4.4)
$$
for all $x\in \mathcal{G}$. Linearizing $(4.4)$ we obtain
$$
\mathfrak{m}(xy+yx)-\mathfrak{m}(x)\mathfrak{m}(y)-\mathfrak{m}(y)\mathfrak{m}(x)\in
\mathcal{Z(G^\prime)} \eqno(4.5)
$$
for all $x,y\in \mathcal{G}$. In addition, by $(4.3)$ it follows
that
$$
\begin{aligned}
\lambda\mathfrak{m}\big([x,y]\big)&=\lambda^2\mathfrak{l}\big([x,y]\big)+\frac{1}{2}\lambda\mu\big([x,y]\big)=
[\lambda\mathfrak{l}(x),\lambda\mathfrak{l}(y)]+\frac{1}{2}\lambda\mu\big([x,y]\big)\\
&=\left[\mathfrak{m}(x)-\frac{1}{2}\mu(x),\mathfrak{m}(y)-\frac{1}{2}\mu(y) \right]+\frac{1}{2}\lambda\mu\big([x,y]\big)\\
&=[\mathfrak{m}(x),\mathfrak{m}(y)]+\frac{1}{2}\lambda\mu\big([x,y]\big)
\end{aligned}$$
Therefore
$$
\lambda\mathfrak{m}\big([x,y]\big)-[\mathfrak{m}(x),\mathfrak{m}(y)]\in
\mathcal{Z(G^\prime)} \eqno(4.6)
$$
for all $x,y\in \mathcal{G}$. Multiplying $(4.5)$ by $\lambda$ and
comparing with $(4.6)$ we arrive at
$$
2\lambda\mathfrak{m}(xy)-(\lambda+1)\mathfrak{m}(x)\mathfrak{m}(y)-(\lambda-1)\mathfrak{m}(y)\mathfrak{m}(x)\in
\mathcal{Z(G^\prime)}
$$
for all $x,y\in \mathcal{G}$. Consequently, the mapping
$$
\varepsilon(x,y):=\lambda\mathfrak{m}(xy)-\frac{1}{2}(\lambda+1)\mathfrak{m}(x)\mathfrak{m}(y)-\frac{1}{2}(\lambda-1)\mathfrak{m}(y)\mathfrak{m}(x)
$$
maps from $\mathcal{G}\times\mathcal{G}$ into
$\mathcal{Z(G^\prime)}$. Let us denote $\frac{1}{2}(\lambda+1)$ by
$\alpha$. Then
$$
\lambda\mathfrak{m}(xy)=\alpha\mathfrak{m}(x)\mathfrak{m}(y)+(\alpha-1)\mathfrak{m}(y)\mathfrak{m}(x)+\varepsilon(x,y)
\eqno(4.7)
$$
for all $x,y\in \mathcal{G}$.

Our aim is to show that $\varepsilon=0$ and that $\alpha=0$ or
$\alpha=1$. In view of $(4.7)$ we have
$$
\begin{aligned}
\lambda^2{\mathfrak m}(xyz)&=\lambda^2{\mathfrak
m}(x(yz))=\lambda\alpha{\mathfrak m}(x){\mathfrak m}(yz)+
\lambda(\alpha-1){\mathfrak m}(yz){\mathfrak m}(x)+\lambda\varepsilon(x,yz)\\
&=\alpha{\mathfrak m}(x)\big(\alpha{\mathfrak m}(y){\mathfrak m}(z)+(\alpha-1){\mathfrak m}(z){\mathfrak m}(y)+\varepsilon(y,z)\big)\\
&\quad +(\alpha-1)\big(\alpha{\mathfrak m}(y){\mathfrak m}(z)+(\alpha-1){\mathfrak m}(z){\mathfrak m}(y)+\varepsilon(y,z)\big){\mathfrak m}(x)+\lambda\varepsilon(x,yz)\\
&=\alpha^2{\mathfrak m}(x){\mathfrak m}(y){\mathfrak m}(z)+\alpha(\alpha-1){\mathfrak m}(x){\mathfrak m}(z){\mathfrak m}(y)+\alpha(\alpha-1){\mathfrak m}(y){\mathfrak m}(z){\mathfrak m}(x)\\
&\quad +(\alpha-1)^2{\mathfrak m}(z){\mathfrak m}(y){\mathfrak
m}(x)+\lambda\varepsilon(x,yz)+\lambda\varepsilon(y,z){\mathfrak
m}(x).
\end{aligned}$$
On the other hand,
$$\begin{aligned}
\lambda^2{\mathfrak m}(xyz)&=\lambda^2{\mathfrak
m}((xy)z)=\lambda\alpha{\mathfrak m}(xy){\mathfrak m}(z)+
\lambda(\alpha-1){\mathfrak m}(z){\mathfrak m}(xy)+\lambda\varepsilon(xy,z)\\
&=\alpha^2{\mathfrak m}(x){\mathfrak m}(y){\mathfrak m}(z)+\alpha(\alpha-1){\mathfrak m}(y){\mathfrak m}(x){\mathfrak m}(z)+\alpha(\alpha-1){\mathfrak m}(z){\mathfrak m}(x){\mathfrak m}(y)\\
&\quad +(\alpha-1)^2{\mathfrak m}(z){\mathfrak m}(y){\mathfrak
m}(x)+\lambda\varepsilon(xy,z)+\lambda\varepsilon(x,y){\mathfrak
m}(z).
\end{aligned}$$
Comparing the above two identities we get
$$
\alpha(\alpha-1)[{\mathfrak m}(y),[{\mathfrak m}(z),{\mathfrak
m}(x)]]+\lambda\varepsilon(y,z){\mathfrak
m}(x)-\lambda\varepsilon(x,y){\mathfrak m}(z)\in
\mathcal{Z(G^\prime)} \eqno(4.8)
$$
for all $x,y,z\in \mathcal{G}$. Substituting $x^2$ for $z$ in
$(4.8)$ and using $(4.4)$ we arrive at
$$
\lambda\varepsilon(y,x^2){\mathfrak
m}(x)-\lambda\varepsilon(x,y){\mathfrak m}(x)^2\in
\mathcal{Z(G^\prime)} \eqno(4.9)
$$
for all $x,y\in \mathcal{G}$. Thus $(4.3)$ can be written as
$$
-\lambda^3\varepsilon(x,y){\mathfrak
l}(x)^2+\lambda^2\big(\varepsilon(y,x^2)+\mu(x)\varepsilon(x,y)\big){\mathfrak
l}(x)\in \mathcal{Z(G^\prime)} \eqno(4.10)
$$
for all $x,y\in \mathcal{G}$, which is due to $(4.3)$. Commuting
with arbitrary $u\in \mathcal{G'}$ and then with $[{\mathfrak
l}(x),u]$ we obtain
$$
\lambda^3\varepsilon(x,y)[[{\mathfrak l}(x)^2,u],[{\mathfrak
l}(x),u]]=0 \eqno(4.11)
$$
for all $x,y\in \mathcal{G}$. We may assume that $A^\prime$ is
non-commutative. Then choose $a_1,a_2\in A^\prime$ such that
$a_1[a_1,a_2]a_1\neq 0$. Putting
$$
{\mathfrak l}(x_0)=\left[
\begin{array}
[c]{cc}%
a_1 & 0\\
0 & 0\\
\end{array}
\right]\quad\text{and}\quad u=\left[
\begin{array}
[c]{cc}%
a_2 & m\\
0 & 0\\
\end{array}
\right]
$$
for some $x_0\in\mathcal{G}$ and an arbitrary $m\in M^\prime$ in
$(4.11)$ gives
$$
\pi_{A^\prime}\big(\lambda^3\varepsilon(x_0,y)\big)a_1[a_1,a_2]a_1m=0
$$
for all $m\in M^\prime$. By the loyality of $M^\prime$ it follows
that
$\pi_{A'}\big(\lambda^3\varepsilon(x_0,y)\big)a_1[a_1,a_2]a_1=0$.
Hence $\pi_{A'}\big(\lambda^3\varepsilon(x_0,y)\big)=0$ by Lemma
\ref{xxsec3.1}. This shows that $\lambda^3\varepsilon(x_0,y)=0$ for
all $y\in \mathcal{G}$. Since $\lambda\neq 0$,
$\varepsilon(x_0,\mathcal{G})=0$ by Lemma \ref{xxsec3.2}. According
to $(4.10)$ we now get $\lambda^2\varepsilon(y,x^2_0){\mathfrak
l}(x_0) \in \mathcal{Z(G^\prime)}$ for all $y\in \mathcal{G}$. This
implies that $\varepsilon(\mathcal{G},x_0^2)=0$. We assert that
$\varepsilon$ is symmetric. Taking $z=x$ into $(4.8)$ and using
$(4.3)$ yields
$$
\lambda^2\big(\varepsilon(y,x)-\varepsilon(x,y)\big){\mathfrak
l}(x)\in \mathcal{Z(G^\prime)} \eqno(4.12)
$$
for all $x,y\in \mathcal{G}$. If $x=x_0$, then
$\lambda^2\varepsilon(y,x_0){\mathfrak l}(x_0)\in
\mathcal{Z(G^\prime)}$ for all $y\in \mathcal{G}$. Thus, similarly
as above, we conclude that $\varepsilon(\mathcal{G},x_0)=0$.
Replacing $x$ by $x+x_0$ in $(4.12)$ we have
$$
\lambda^2\big(\varepsilon(y,x)-\varepsilon(x,y)\big){\mathfrak
l}(x_0)\in \mathcal{Z(G^\prime)}
$$
for all $x,y\in \mathcal{G}$. This implies that $\varepsilon$ is
symmetric. Replacing $x$ by $y\pm x_0$ in $(4.9)$ and combining
those two relations we get
$$
2\lambda\varepsilon(y,x_0\circ y){\mathfrak
m}(x_0)-2\lambda\varepsilon(y,y){\mathfrak m}(x_0)^2\in
\mathcal{Z(G^\prime)}
$$
for all $y\in \mathcal{G}$, which can be in view of $(4.3)$ written
as
$$
-\lambda^3\varepsilon(y,y){\mathfrak
m}(x_0)^2+\lambda^2\big(\varepsilon(y,x_0\circ
y)-\mu(x_0)\varepsilon(y,y)\big){\mathfrak l}(x_0)\in
\mathcal{Z(G^\prime)}
$$
for all $y\in \mathcal{G}$. Therefore
$$
\lambda^3\varepsilon(y,y)[[{\mathfrak m}(x_0)^2,u],[{\mathfrak
m}(x_0),u]]=0
$$
for all $y\in \mathcal{G}$ and $u\in \mathcal{G'}$. Similarly as
above it follows that $\lambda^3\varepsilon(y,y)=0$ and hence
$\varepsilon(y,y)=0$ for all $y\in \mathcal{G}$. The linearization
of $\varepsilon(y,y)=0$ shows that $\varepsilon=0$. Correspondingly,
$(4.8)$ gives
$$
\lambda^4\alpha(\alpha-1)[{\mathfrak l}(x),[{\mathfrak
l}(y),[{\mathfrak l}(z), {\mathfrak l}(w)]]]=0
$$
for all $x,y,z,w\in \mathcal{G}$, Since $\mathfrak{l}$ is
surjective, we know that
$\lambda^4\alpha(\alpha-1)[x',[y',[z',w']]]=0$ for all
$x',y',z',w'\in \mathcal{G^\prime}$. Let us take
$$
x'=y'=z'=\left[
\begin{array}
[c]{cc}%
1 & 0\\
0 & 0\\
\end{array}
\right]\quad\text{and}\quad w'=\left[
\begin{array}
[c]{cc}%
0 & m\\
0 & 0\\
\end{array}
\right],
$$
where $m$ is an arbitrary element in $M^\prime$. Thus
$\pi_{A'}(\lambda^4\alpha(\alpha-1))m=0$ for all $m\in M'$.
Therefore $\pi_{A'}(\lambda^4\alpha(\alpha-1))=0$ and hence
$\lambda^4\alpha(\alpha-1)=0$. Using Lemma \ref{xxsec3.2} we see
that $\alpha=0$ or $\alpha=1$.

Assume that $\alpha=0$. Then $\lambda=2\alpha-1=-1$, which by
$(4.7)$ further implies that $\mathfrak{m}$ is an anti-homomorphism.
Let us write $\mathfrak{n}(x)=\mu(x)/2$. It follows from $(4.3)$
that $\mathfrak{l}=-\mathfrak{m}+\mathfrak{n}$, which clearly yields
that $\mathfrak{n}([x,y])=0$ for all $x,y\in \mathcal{G}$. In an
analogous way we claim that if $\alpha=1$, then
$\mathfrak{l}=\mathfrak{m}+\mathfrak{n}$, where $\mathfrak{m}$ is a
homomorphism and $\mathfrak{n}(x)=-\mu(x)/2$ vanishing on each
commutator.

We next show that $\mathfrak{m}$ is injective. Suppose that
$\mathfrak{m}(x)=0$ for some $x\in\mathcal{G}$. Then
$\mathfrak{l}(x)\in \mathcal{Z(G^\prime)}$ and hence $x\in
\mathcal{Z(G)}$. Thus $\ker(\mathfrak{m})\subseteq \mathcal{Z(G)}$.
Note that the generalized matrix algebra $\mathcal{G}$ does not
contain nonzero central ideals (see Lemma \ref{xxsec3.3}). So
$\ker(\mathfrak{m})=0$.

Finally, we need to prove that if $\mathcal{G^\prime}$ is central
over $\mathcal{R}$, then $\mathfrak{m}$ is surjective. We claim that
$\mathfrak{m}(1)=1$ or $\mathfrak{m}(1)=-1$. Since $\mathfrak{l}$ is
a Lie isomorphism, we have $\mathfrak{l}(1)\in
\mathcal{Z(G^\prime)}$ and hence
$\mathfrak{m}(1)=\mathfrak{l}(1)-\mathfrak{n}(1)\in
\mathcal{Z(G^\prime)}$. Further, in the case $\mathfrak{m}$ is a
homomorphism, we have
$\mathfrak{m}(x)=\mathfrak{m}(1)\mathfrak{m}(x)$ for all $x\in
\mathcal{G}$. Using
$\mathfrak{m}(x)=\mathfrak{l}(x)-\mathfrak{n}(x)$ we get
$(\mathfrak{m}(1)-1)\mathfrak{l}(x)-(\mathfrak{m}(1)-1)\mathfrak{n}(x)=0$
for all $x\in \mathcal{G}$. Hence
$(\mathfrak{m}(1)-1)[\mathcal{G'},\mathcal{G'}]=0$. Applying Lemma
\ref{xxsec3.1} yields that $\mathfrak{m}(1)=1$. Similarly, if
$\mathfrak{m}$ is the negative of an anti-homomorphism, we obtain
$\mathfrak{m}(1)=-1$. We may write
$\mathfrak{n}(x)=\mathfrak{h}(x)1$ for some linear mapping
$\mathfrak{h}: \mathcal{G}\longrightarrow\mathcal{R}$. Therefore
$\mathfrak{l}(x)=\mathfrak{m}(x)+\mathfrak{h}(x)1=\mathfrak{m}(x\pm
\mathfrak{h}(x)1)$. So $\mathfrak{m}$ is surjective, which is due to
the fact $\mathfrak{l}$ is bijective. Thus we complete the proof of
the proposition.
\end{proof}

\begin{theorem}\label{xxsec4.3}
Let $\mathcal{G}=\left[
\begin{array}
[c]{cc}%
A & M\\
N & B\\
\end{array}
\right]$ and $\mathcal{G}^\prime=\left[
\begin{array}
[c]{cc}%
A^\prime & M^\prime\\
N^\prime & B^\prime\\
\end{array}
\right]$ be generalized matrix algebras over $\mathcal{R}$ with
$1/2\in\mathcal{R}$. Let $\mathfrak{l}\colon
\mathcal{G}\longrightarrow\mathcal{G'}$ be a Lie isomorphism. If
\begin{enumerate}
\item[{\rm(1)}] every commuting linear mapping on $A'$ or $B'$ is proper,
\item[{\rm(2)}] $\pi_{A'}(\mathcal{Z(G^\prime)})={\mathcal Z}(A')\neq A'$
and $\pi_{B'}(\mathcal{Z(G^\prime)})={\mathcal Z}(B')\neq B'$,
\item[{\rm(3)}] either $A$ or $B$ is noncommutative,
\item[{\rm(4)}] $M'$ is loyal,
\end{enumerate}
then $\mathfrak{l}=\mathfrak{m}+\mathfrak{n}$, where
$\mathfrak{m}\colon: \mathcal{G}\longrightarrow\mathcal{G^\prime}$
is a homomorphism or the negative of an anti-homomorphism,
$\mathfrak{m}$ is injective, and $\mathfrak{n}\colon
\mathcal{G}\longrightarrow {\mathcal Z}(\mathcal{G}^\prime)$ is a
linear mapping vanishing on each commutator. Moreover, if
$\mathcal{G^\prime}$ is central over $\mathcal{R}$, then
$\mathfrak{m}$ is surjective.
\end{theorem}

\begin{proof}
It follows from Theorem \ref{xxsec3.4} and Proposition
\ref{xxsec4.2} directly.
\end{proof}

As a direct consequence of Theorem \ref{xxsec4.3} we have

\begin{corollary}\cite[Theorem 4.3]{BenkovicEremita}\label{xxsec4.4}
Let $\mathcal{G}=\left[
\begin{array}
[c]{cc}%
A & M\\
O & B\\
\end{array}
\right]$ and $\mathcal{G}^\prime=\left[
\begin{array}
[c]{cc}%
A^\prime & M^\prime\\
O & B^\prime\\
\end{array}
\right]$ be triangular algebras over $\mathcal{R}$ with
$1/2\in\mathcal{R}$. Let $\mathfrak{l}\colon
\mathcal{G}\longrightarrow\mathcal{G'}$ be a Lie isomorphism. If
\begin{enumerate}
\item[{\rm(1)}] every commuting linear mapping on $A'$ or $B'$ is proper,
\item[{\rm(2)}] $\pi_{A'}(\mathcal{Z(G^\prime)})={\mathcal Z}(A')\neq A'$ and $\pi_{B'}(\mathcal{Z(G^\prime)})={\mathcal Z}(B')\neq B'$,
\item[{\rm(3)}] either $A$ or $B$ is noncommutative,
\item[{\rm(4)}] $M'$ is loyal,
\end{enumerate}
then $\mathfrak{l}=\mathfrak{m}+\mathfrak{n}$, where
$\mathfrak{m}\colon: \mathcal{G}\longrightarrow\mathcal{G^\prime}$
is a homomorphism or the negative of an anti-homomorphism,
$\mathfrak{m}$ is injective, and $\mathfrak{n}\colon
\mathcal{G}\longrightarrow {\mathcal Z}(\mathcal{G}^\prime)$ is a
linear mapping vanishing on each commutator. Moreover, if
$\mathcal{G^\prime}$ is central over $\mathcal{R}$, then
$\mathfrak{m}$ is surjective.
\end{corollary}

In particular, we also have

\begin{corollary}\label{xxsec4.5}\cite[Corollary 4.4]{BenkovicEremita}
Let $n\geq 2$ and $\mathcal{R}$ be a commutative domain with
$\frac{1}{2}\in \mathcal{R}$. If ${\mathfrak l}\colon {\mathcal
T}_n(\mathcal{R})\longrightarrow {\mathcal T}_n(\mathcal{R})$ is a
Lie isomorphism, then ${\mathfrak l}={\mathfrak m}+ {\mathfrak n}$,
where ${\mathfrak m}\colon {\mathcal T}_n({\mathcal
R})\longrightarrow {\mathcal T}_n(\mathcal{R})$ is an isomorphism or
the negative of an antiisomorphism and ${\mathfrak n}\colon
{\mathcal T}_n(\mathcal{R})\longrightarrow \mathcal{R}1$ is a linear
mapping vanishing on each commutator.
\end{corollary}

\begin{corollary}\label{xxsec4.6}\cite[Corollary 4.5]{BenkovicEremita}
Let $\mathcal{N}$ and $\mathcal{N}^\prime $ be nests on a Hilbert
space $\mathbf{H}$, ${\mathcal T}(\mathcal{N})$ and ${\mathcal
T}(\mathcal{N}^\prime)$ be the nest algebras associated with
$\mathcal{N}$ and $\mathcal{N}^\prime$, respectively. If ${\mathfrak
l}\colon {\mathcal T}(\mathcal{N})\longrightarrow {\mathcal
T}(\mathcal{N}^\prime)$ is a Lie isomorphism, then ${\mathfrak
l}={\mathfrak m}+{\mathfrak n}$, where ${\mathfrak m}\colon
{\mathcal T}({\mathcal N})\longrightarrow {\mathcal
T}(\mathcal{N}^\prime)$ is an isomorphism or the negative of an
antiisomorphism and ${\mathfrak n}\colon {\mathcal
T}(\mathcal{N})\longrightarrow \mathbb{C}1^\prime$ is a linear
mapping vanishing on each commutator.
\end{corollary}

For the Lie isomorphisms of full matrix algebras, we have similar
characterizations.

\begin{corollary}\label{xxsec4.7}
Let $\mathcal{R}$ be a commutative domain with $\frac{1}{2}\in
\mathcal{R}$. If $\mathfrak{l}\colon {\mathcal
M}_n(\mathcal{R})\rightarrow {\mathcal M}_n(\mathcal{R})$ ($n\geq
3$) is a Lie isomorphism, then
$\mathfrak{l}=\mathfrak{m}+\mathfrak{n}$, where $\mathfrak{m}\colon
{\mathcal M}_n(\mathcal{R})\rightarrow {\mathcal M}_n(\mathcal{R})$
is an isomorphism or the negative of an anti-isomorphism and
$\mathfrak{n}\colon {\mathcal M}_n(\mathcal{R})\rightarrow {\mathcal
R}1$ is a linear mapping vanishing on each commutator.
\end{corollary}

\begin{proof}
We write $\mathcal{M}_n(\mathcal{R})=\left[
\begin{array}
[c]{cc}%
\mathcal{R} & {\mathcal M}_{1\times (n-1)}(\mathcal{R})\\
\mathcal{M}_{(n-1)\times 1}(\mathcal{R}) & \mathcal{M}_{n-1}(\mathcal{R})\\
\end{array}
\right]$. Corollary \ref{xxsec3.16} shows that each commuting trace
of arbitrary bilinear mapping on ${\mathcal M}_n(\mathcal{R})$ is
proper. Moreover, ${\mathcal M}_{n-1}(\mathcal{R})$ is
noncommutative and ${\mathcal M}_{1\times (n-1)}(\mathcal{R})$ is a
loyal $(\mathcal{R},{\mathcal M}_{n-1}(\mathcal{R}))$-bimodule.
Hence Proposition \ref{xxsec4.2} implies the conclusion.
\end{proof}

\begin{corollary}\label{xxsec4.8}
Let $\mathcal{R}$ be a commutative domain with $\frac{1}{2}\in
\mathcal{R}$, $V$ be an $\mathcal{R}$-linear space and
$B(\mathcal{R}, V, \gamma)$ be the inflated algebra of $\mathcal{R}$
along $V$. If $\mathfrak{l}\colon B(\mathcal{R}, V,
\gamma)\longrightarrow B(\mathcal{R}, V, \gamma)$ is a Lie
isomorphism, then $\mathfrak{l}=\mathfrak{m}+\mathfrak{n}$, where
$\mathfrak{m}\colon B(\mathcal{R}, V, \gamma)\longrightarrow
B(\mathcal{R}, V, \gamma)$ is an isomorphism or the negative of an
anti-isomorphism and $\mathfrak{n}\colon B(\mathcal{R}, V,
\gamma)\longrightarrow {\mathcal R}1$ is a linear mapping vanishing
on each commutator.
\end{corollary}

Let us consider the Lie isomorphisms of several unital algebras with
nontrivial idempotents.

\begin{corollary}\label{xxsec4.9}
Let $\mathcal{A}$ be a unital prime algebra with nontrivial
idempotent and ${\mathfrak l}\colon \mathcal{A}\longrightarrow
\mathcal{A}$ be a Lie isomorphism. Then every Lie isomorphism is of
the standard form $(\spadesuit)$.
\end{corollary}

\begin{corollary}\label{xxsec4.10}
Let $X$ be a Banach space over the real or complex field
$\mathbb{F}$, $\mathcal{B}(X)$ be the algebra of all bounded linear
operators on $X$. Then every Lie isomorphism has the standard form
$(\spadesuit)$.
\end{corollary}

\section{Potential Topics for Further Research}\label{xxsec5}

Although the main goal of the current article is to consider
commuting traces and Lie isomorphisms on generalized matrix
algebras, there are more interesting mappings related to our current
work on generalized matrix algebras. These mappings are still
considerable interest and will draw more people's our attention. In
this section we will propose several potential topics for future
further research.

Let $\mathcal{R}$ be a commutative ring with identity, $\mathcal{A}$
be a unital algebra over $\mathcal{R}$ and $\mathcal{Z(A)}$ be the
center of $\mathcal{A}$. Recall that an $\mathcal{R}$-linear mapping
${\mathfrak f}: \mathcal{A}\longrightarrow \mathcal{A}$ is said to
be \textit{centralizing} if $[{\mathfrak f}(a), a]\in
\mathcal{Z(A)}$ for all $a\in \mathcal{A}$. Let $n$ be a positive
integer and $\mathfrak{q}\colon \mathcal{A}^n\longrightarrow
\mathcal{A}$ be an $n$-linear mapping. The mapping ${\mathfrak
T}_{\mathfrak q}\colon \mathcal{A}\longrightarrow \mathcal{A}$
defined by ${\mathfrak T}_{\mathfrak q}(a)={\mathfrak q}(a, a,
\cdots, a)$ is called a \textit{trace} of ${\mathfrak q}$. We say
that a centralizing trace ${\mathfrak T}_{\mathfrak q}$ is
\textit{proper} if it can be written as
$$
{\mathfrak T}_{\mathfrak q}(a)=\sum_{i=0}^n\mu_i(a)a^{n-i},
\hspace{8pt} \forall a\in \mathcal{A},
$$
where $\mu_i(0\leq i\leq n)$ is a mapping from $\mathcal{A}$ into
$\mathcal{Z(A)}$ and every $\mu_i(0\leq i\leq n)$ is in fact a trace
of an $i$-linear mapping ${\mathfrak q}_i$ from $\mathcal{A}^i$ into
$\mathcal{Z(A)}$. Let $n=1$ and ${\mathfrak f}\colon
A\longrightarrow A$ be an $\mathcal{R}$-linear mapping. In this
case, an arbitrary trace ${\mathfrak T}_{\mathfrak f}$ of
${\mathfrak f}$ exactly equals to itself. Moreover, if a
centralizing trace ${\mathfrak T}_{\mathfrak f}$ of ${\mathfrak f}$
is proper, then it has the form
$$
{\mathfrak T}_{\mathfrak f}(a)=z a+\eta(a), \hspace{8pt} \forall
a\in \mathcal{A},
$$
where $z\in \mathcal{Z(A)}$ and $\eta$ is an $\mathcal{R}$-linear
mapping from $A$ into $\mathcal{Z(A)}$. Let us see the case of
$n=2$. Suppose that ${\mathfrak g}\colon \mathcal{A}\times
\mathcal{A}\longrightarrow \mathcal{A}$ is an $\mathcal{R}$-bilinear
mapping. If a centralizing trace ${\mathfrak T}_{\mathfrak g}$ of
${\mathfrak g}$ is proper, then it is of the form
$$
{\mathfrak T}_{\mathfrak g}(a)=z a^2+\mu(a)a+\nu(a), \hspace{8pt}
\forall a\in \mathcal{A},
$$
where $z\in \mathcal{Z(A)}$, $\mu$ is an $\mathcal{R}$-linear
mapping from $A$ into $\mathcal{Z(A)}$ and $\nu$ is a trace of some
bilinear mapping. Bre\v{s}ar started the study
of commuting and centralizing traces of multilinear mappings in his
series of works \cite{Bresar0, Bresar1, Bresar2, Bresar3}, where he
investigated the structure of centralizing traces of (bi-)linear
mappings on prime rings. It has turned out that in certain rings, in
particular, prime rings of characteristic different from $2$ and
$3$, every centralizing trace of a biadditive mapping is commuting.
Moreover, every centralizing mapping of a prime ring of
characteristic not $2$ is of the proper form and is actually
commuting. Lee et al further generalized Bre\v{s}ar's results by
showing that each commuting trace of an arbitrary multilinear
mapping on a prime ring also has the proper form
\cite{LeeWongLinWang}. An exciting discovery is that every
centralizing trace of arbitrary bilinear mapping on triangular
algebras is commuting in some additional conditions.

\begin{theorem}\cite[Theorem 3.4]{XiaoWei2}\label{xxsec5.1}
Let $\mathcal{T}=\left[
\begin{array}
[c]{cc}%
A & M\\
O & B\\
\end{array}
\right]$ be a $2$-torsionfree triangular algebras over a commutative
ring $\mathcal{R}$ and $\mathfrak{q}\colon \mathcal{T}\times
\mathcal{T}\longrightarrow \mathcal{T}$ be an $\mathcal{R}$-bilinear
mapping. If
\begin{enumerate}
\item[(1)] every commuting linear mapping on $A$ or $B$ is proper,
\item[(2)] $\pi_A({\mathcal Z}(\mathcal{T}))={\mathcal Z}(A) \neq A $ and
$\pi_B({\mathcal Z}(\mathcal{T}))={\mathcal Z}(B)\neq B$,
\item[(3)] $M$ is loyal,
\end{enumerate}
then every centralizing trace ${\mathfrak T}_{\mathfrak q}:
\mathcal{T}\longrightarrow \mathcal{T}$ of ${\mathfrak q}$ is
proper. Moreover, each centralizing trace $\mathfrak{T_q}$ of
$\mathfrak{q}$ is commuting.
\end{theorem}

It is natural to formulate the following question

\begin{question}\label{xxsec5.2}
Let $\mathcal{T}=\left[
\begin{array}
[c]{cc}%
A & M\\
O & B\\
\end{array}
\right]$ be a $2$-torsionfree triangular algebra over a commutative
ring $\mathcal{R}$ and ${\mathfrak q}\colon \mathcal{T}\times
\mathcal{T} \times \cdots \times \mathcal{T}\longrightarrow
\mathcal{T}$ be an $n$-linear mapping. Suppose that the following
conditions are satisfied
\begin{enumerate}
\item[{\rm(1)}] each commuting linear mapping on $A$ or $B$ is proper;
\item[{\rm(2)}] $\pi_A(\mathcal{Z(G)})=\mathcal{Z}(A)\neq A$ and
$\pi_B(\mathcal{Z(G)})=\mathcal{Z}(B)\neq B$;
\item[{\rm(3)}] $M$ is loyal.
\end{enumerate}
Is any centralizing trace ${\mathfrak T}_{\mathfrak q}:
\mathcal{T}\longrightarrow \mathcal{T}$ of ${\mathfrak q}$ proper ?
Furthermore, what can we say about the centralizing traces of
multilinear mappings on a generalized matrix algebra
$\mathcal{G}=\left[
\begin{array}
[c]{cc}%
A & M\\
N & B\\
\end{array}
\right]$ ?
\end{question}

Let $\mathcal{R}$ be a commutative ring with identity, $\mathcal{A}$
and $\mathcal{B}$ be associative $\mathcal{R}$-algebras. We define a
\textit{Lie triple isomorphism} from $\mathcal{A}$ into
$\mathcal{B}$ to be an $\mathcal{R}$-linear bijective mapping
${\mathfrak l}$ satisfying the condition
$$
{\mathfrak l}([[a, b],c])=[[{\mathfrak l}(a), {\mathfrak
l}(b)],\mathfrak{l}(c)], \hspace{8pt} \forall a, b, c\in
\mathcal{A}.
$$
Obviously, every Lie isomorphism is a Lie triple isomorphism. The
converse is, in general, not true. In \cite{XiaoWei2} we apply
Theorem \ref{xxsec5.1} to the study of Lie triple isomorphisms on
triangular algebras. It is shown that every Lie triple isomorphism
between triangular algebras also has an \textit{approximate standard
decomposition expression} under some additional conditions. That is

\begin{theorem}\cite[Theorem 4.3]{XiaoWei2}\label{xxsec5.3}
Let $\mathcal{T}=\left[
\begin{array}
[c]{cc}%
A & M\\
O & B\\
\end{array}
\right]$ and $\mathcal{T}^\prime=\left[
\begin{array}
[c]{cc}%
A^\prime & M^\prime\\
O & B^\prime\\
\end{array}
\right]$ be triangular algebras over $\mathcal{R}$ with
$1/2\in\mathcal{R}$. Let $\mathfrak{l}\colon
\mathcal{T}\longrightarrow\mathcal{T'}$ be a Lie triple isomorphism.
If
\begin{enumerate}
\item[{\rm(1)}] every commuting linear mapping on $A'$ or $B'$ is proper,
\item[{\rm(2)}] $\pi_{A'}(\mathcal{Z(T^\prime)})={\mathcal Z}(A')\neq A'$ and
$\pi_{B'}(\mathcal{Z(T^\prime)})={\mathcal Z}(B')\neq B'$,
\item[{\rm(3)}] either $A$ or $B$ is noncommutative,
\item[{\rm(4)}] $M'$ is loyal,
\end{enumerate}
then $\mathfrak{l}=\pm\mathfrak{m}+\mathfrak{n}$, where
$\mathfrak{m}: \mathcal{T}\longrightarrow\mathcal{T^\prime}$
is a Jordan homomorphism, $\mathfrak{m}$ is injective, and
$\mathfrak{n}\colon \mathcal{T}\longrightarrow {\mathcal
Z}(\mathcal{T}^\prime)$ is a linear mapping vanishing on each second
commutator. Moreover, if $\mathcal{T^\prime}$ is central over
$\mathcal{R}$, then $\mathfrak{m}$ is surjective.
\end{theorem}

A question closely related to the above theorem is

\begin{question}\label{xxsec5.4}
Let $\mathcal{G}=\left[
\begin{array}
[c]{cc}%
A & M\\
N & B\\
\end{array}
\right]$ and $\mathcal{G}^\prime=\left[
\begin{array}
[c]{cc}%
A^\prime & M^\prime\\
N^\prime & B^\prime\\
\end{array}
\right]$ be generalized matrix algebras over $\mathcal{R}$ with
$1/2\in\mathcal{R}$. Let $\mathfrak{l}\colon
\mathcal{G}\longrightarrow\mathcal{G^\prime}$ be a Lie triple
isomorphism. Under what conditions does $\mathfrak{l}$ has a similar
decomposition expression ?
\end{question}

\bigskip



\begin{thebibliography}{}


\bibitem[1]{BaiDuHou1} Z. -F. Bai, S. -P. Du and J. -C. Hou, {\em Multiplicative Lie
isomorphisms between prime rings}, Comm. Algebra, \textbf{36}
(2008), 1626-1633.

\bibitem[2]{BaiDuHou2} Z. -F. Bai, S. -P. Du and J. -C. Hou, {\em Multiplicative
$\ast$-Lie isomorphisms between factors}, J. Math. Anal. Appl.,
\textbf{346} (2008), 327-335.

\bibitem[3]{BanningMathieu} R. Banning and M. Mathieu, {\em Commutativity preserving mappings on
semiprime rings}, Comm. Algebra, \textbf{25} (1997), 247-265.

\bibitem[4]{BeidarBresarChebotar} K. I. Beidar, M. Bre\v{s}ar and
M. A. Chebotar, {\em Functional identities on upper triangular
matrix algebras}, J. Math. Sci., \textbf{102} (2000) 4557-4565.

\bibitem[5]{BeidarBresarChebotarMartindale} K. I. Beidar, M. Bre\v{s}ar,
M. A. Chebotar and W. S. Martindale, 3rd, {\em On Herstein's Lie Map
Conjectures, III}, J. Algebra, \textbf{249} (2002), 59-94.

\bibitem[6]{BeidarMartindaleMikhalev} K. I. Beidar, W. S.
Martindale, 3rd and A. V. Mikhalev, {\em Lie isomorphisms in prime
rings with involution}, J. Algebra, \textbf{169} (1994), 304-327.

\bibitem[7]{BenkovicEremita} D. Benkovi$\check{\rm c}$ and D. Eremita, {\em
Commuting traces and commmutativity preserving maps on triangular
algebras}, J. Algebra, \textbf{280} (2004), 797-824.

\bibitem[8]{Blau1} P. S. Blau, {\em Lie isomorphisms of prime rings},
Ph.D. Thesis, University of Massachusetts Amherst, 1996, 68 pp.

\bibitem[9]{Blau2} P. S. Blau, {\em Lie isomorphisms of non-GPI
rings with involution}, Comm. Algebra, \textbf{27} (1999),
2345-2373.

\bibitem[10]{Bresar0} M. Bre$\check{\rm s}$ar, {\em On a generalization
of the notion of centralizing mappings}, Proc. Amer. Math. Soc.,
\textbf{114} (1992), 641-649.

\bibitem[11]{Bresar1} M. Bre$\check{\rm s}$ar, {\em Commuting
traces of biadditive mappings, commutativity-preserving mappings and
Lie mappings}, Trans. Amer. Math. Soc., \textbf{335} (1993),
525-546.

\bibitem[12]{Bresar2} M. Bre$\check{\rm s}$ar, {\em Centralizing mappings
and derivations in prime rings}, J. Algebra, \textbf{156} (1993),
385-394.

\bibitem[13]{Bresar3} M. Bre$\check{\rm s}$ar, {\em Commuting maps:
a survey}, Taiwanese J. Math., \textbf{8} (2004), 361-397.


\bibitem[14]{BresarSemrl} M. Bre\v{s}ar and P. \v{S}emrl, {\em Commuting traces
of biadditive maps revisited}, Comm. Algebra, \textbf{31} (2003),
381-388.

\bibitem[15]{Brown} W. P. Brown, {\em Generalized matrix algebras},
Canad. J. Math., \textbf{7} (1955), 188-190.

\bibitem[16]{CalderonGonzalez1} A. J. Calder\'{o}n Mart\'{i}n and C.
Mart\'{i}n Gonz\'{a}lez, {\em Lie isomorphisms on $H^*$-algebras},
Comm. Algebra, \textbf{31} (2003), 323-333.

\bibitem[17]{CalderonGonzalez2} A. J. Calder\'{o}n Mart\'{i}n and C.
Mart\'{i}n Gonz\'{a}lez, {\em The Banach-Lie group of Lie triple
automorphisms of an $H^*$-algebra}, Acta Math. Sci. (Ser. English),
\textbf{30} (2010), 1219-1226.

\bibitem[18]{CalderonGonzalez3} A. J. Calder\'{o}n Mart\'{i}n and C.
Mart\'{i}n Gonz\'{a}lez, {\em A linear approach to Lie triple
automorphisms of $H^*$-algebras}, J. Korean Math. Soc., \textbf{48}
(2011), 117-132.

\bibitem[19]{CalderonHaralampidou} A. J. Calder\'{o}n Mart\'{i}n and M.
Haralampidou, {\em Lie mappings on locally $m$-convex
$H^*$-algebras}, International Conference on Topological Algebras
and their Applications. ICTAA 2008, 42-51, Math. Stud. (Tartu), 4,
Est. Math. Soc., Tartu, 2008.

\bibitem[20]{Cheung1} W. S. Cheung, {\em Maps on triangular algebras},
Ph.D. Dissertation, University of Victoria, 2000. 172pp.

\bibitem[21]{Cheung2} W. S. Cheung, {\em Commuting maps of triangular algebras},
J. London Math. Soc., \textbf{63} (2001), 117-127.

\bibitem[22]{Cheung3} W. S. Cheung, {\em Lie derivations of triangular
algebras}, Linear Multilinear Algebra, \textbf{51} (2003), 299-310.

\bibitem[23]{Dokovic} D. \v{Z}. Dokovi{\'c}, {\em Automorphisms
of the Lie algebra of upper triangular matrices over a connected
commutative ring}, J. Algebra, \textbf{170} (1994), 101-110.

\bibitem[24]{Dolinar1} G. Dolinar, {\em Maps on upper triangular matrices preserving Lie
products}, Linear Multilinear Algebra, \textbf{55} (2007), 191-198.

\bibitem[25]{Dolinar2} G. Dolinar, {\em Maps on $M_n$ preserving Lie products},
Publ. Math. Debrecen, \textbf{71} (2007), 467-477.

\bibitem[26]{DuWang} Y. Du and Y. Wang, {\em Lie derivations of
generalized matrix algebras}, Linear Algebra Appl., \textbf{437}
(2012) 2719-2726.

\bibitem[27]{Herstein} I. N. Herstein, {\em Lie and Jordan structures in
simple, associative rings}, Bull. Amer. Math. Soc., \textbf{67}
(1961), 517-531.

\bibitem[28]{Hua} L. Hua, {\em A theorem on matrices over an sfield and its applications},
J. Chinese Math. Soc. (N.S.), \textbf{1} (1951), 110-163.

\bibitem[29]{Krylov} P. A. Krylov, {\em Isomorphism of genralized
matrix rings}, Algebra and Logic, \textbf{47} (2008), 258-262.

\bibitem[30]{LeeWongLinWang} P. -H. Lee, T. -L. Wong, J. -S. Lin
and R. -J. Wang, {\em Commuting traces of multiadditive mappings},
J. Algebra, \textbf{193} (1997), 709-723.

\bibitem[31]{LiWei} Y. -B. Li and F. Wei, {\em Semi-centralizing maps
of genralized matrix algebras}, Linear Algebra Appl., \textbf{436}
(2012), 1122-1153.

\bibitem[32]{LiWykWei} Y. -B. Li, L. van Wyk and F. Wei, {\em Jordan
derivations and antiderivations of genralized matrix algebras},
Oper. Matrices, In Press.

\bibitem[33]{XiaoWei2} X.-F. Liang, Z.-K. Xiao and F. Wei, {\em Centralizing traces and Lie
triple isomorphisms on triangular algebras}, Preprint.

\bibitem[34]{Lu} F. -Y. Lu, {\em Lie isomorphisms of reflexive
algebras}, J. Funct. Anal., \textbf{240} (2006), 84-104.

\bibitem[35]{MarcouxSourour1} L. W. Marcoux and A. R. Sourour, {\em Commutativity
preserving maps and Lie automorphisms of triangular matrix
algebras}, Linear Algebra Appl., \textbf{288} (1999), 89-104.

\bibitem[36]{MarcouxSourour2} L. W. Marcoux and A. R. Sourour, {\em
Lie isomorphisms of nest algebras}, J. Funct. Anal., \textbf{164}
(1999), 163-180.

\bibitem[37]{Martindale1} W. S. Martindale, 3rd, {\em Lie isomorphisms of primitive rings},
Proc. Amer. Math. Soc., \textbf{14} (1963), 909-916.

\bibitem[38]{Martindale2} W. S. Martindale, 3rd, {\em Lie isomorphisms
of simple rings}, J. London Math. Soc., \textbf{44} (1969), 213-221.

\bibitem[39]{Martindale3} W. S. Martindale, 3rd, {\em Prime rings satisfying a generalized
polynomial identity}, J. Algebra, \textbf{12} (1969), 576-584.

\bibitem[40]{Martindale4} W. S. Martindale, 3rd, {\em Lie isomorphisms of prime rings},
Trans. Amer. Math. Soc., \textbf{142} (1969), 437-455.

\bibitem[41]{Martindale5} W. S. Martindale, 3rd, {\em Lie and
Jordan mappings in associative rings}, Ring theory (Proc. Conf.,
Ohio Univ., Athens, Ohio, 1976), pp. 71-84. Lecture Notes in Pure
and Appl. Math., Vol. \textbf{25}, Dekker, New York, 1977.

\bibitem[42]{Martindale6} W. S. Martindale, 3rd, {\em Lie and
Jordan mappings}, Contemporary Math., \textbf{13} (1982), 173-177.

\bibitem[43]{Martindale7} W. S. Martindale, 3rd, {\em Lie maps
in prime rings: a personal perspective}, Rings and Nearrings,
95-110, Walter de Gruyter, Berlin, 2007.

\bibitem[44]{Mathieu} M. Mathieu, {\em Lie mappings of $C^\ast$-algebras},
 Nonassociative algebra and its applications, 229-234,
 Lecture Notes in Pure and Appl. Math., \textbf{211},
 Dekker, New York, 2000.

\bibitem[45]{Miers1} C. R. Miers, {\em Lie isomorphisms of factors},
Trans. Amer. Math. Soc., \textbf{147} (1970), 55-63.

\bibitem[46]{Miers2} C. R. Miers, {\em Lie homomorphsism of operator
algebras}, Pacific J. Math., \textbf{38} (1971), 717-735.

\bibitem[47]{Miers3} C. R. Miers, {\em Lie triple derivations of von Neumann algebras},
Proc. Amer. Math. Soc., \textbf{71} (1978), 57-61.

\bibitem[48]{Morita} K. Morita, {\em Duality for modules and its
applications to the theory of rings with minimum condition}, Sci.
Rep. Tokyo Kyoiku Diagaku Sect. A, \textbf{6} (1958), 83-142.

\bibitem[49]{QiHou1} X. -F. Qi and J. -C. Hou, {\em Characterization of $\xi$-Lie
multiplicative isomorphisms}, Oper. Matrices, \textbf{4} (2010),
417-429.

\bibitem[50]{QiHou2} X. -F. Qi and J. -C. Hou, {\em Characterization of Lie
multiplicative isomorphisms between nest algebras}, Sci. China
Math., \textbf{54} (2011), 2453-2462.

\bibitem[51]{Rosen} M. P. Rosen, {\em Isomorphisms of a certain
class of prime Lie rings}, J. Algebra, \textbf{89} (1984), 291-317.

\bibitem[52]{Semrl} P. \v{S}emrl, {\em Non-linear commutativity preserving maps},
Acta Sci. Math. (Szeged), \textbf{71} (2005), 781-819.

\bibitem[53]{Sourour} A. R. Sourour, {\em Maps on triangular matrix algebras},
Problems in applied mathematics and computational intelligence,
92-96, Math. Comput. Sci. Eng., World Sci. Eng. Soc. Press, Athens,
2001.

\bibitem[54]{WangLu} T. Wang and F. -Y. Lu, {\em Lie isomorphisms of
nest algebras on Banach spaces}, J. Math. Anal. Appl., \textbf{391}
(2012), 582-594.

\bibitem[55]{XiaoWei1} Z. -K. Xiao and F. Wei, {\em Commuting mappings of
generalized matrix algebras}, Linear Algebra Apll., \textbf{433}
(2010), 2178-2197.


\bibitem[56]{YuLu} X. -P. Yu and F. -Y Lu, {\em Maps preserving Lie product
on $B(X)$}, Taiwanese J. Math., \textbf{12} (2008), 793-806.

\bibitem[57]{ZhangZhang} J. -H. Zhang and F. -J. Zhang, {\em Nonlinear maps preserving Lie
products on factor von Neumann algebras}, Linear Algebra Appl.,
\textbf{429} (2008), 18-30.

\end{thebibliography}
\end{document}